\documentclass[11pt,a4paper]{article}

\usepackage{inputenc}
\usepackage{amsmath}
\usepackage{bm}
\usepackage{bbold}
\usepackage{amsthm}
\usepackage{enumerate}

\usepackage{pgfgantt}

\usepackage[top=1.0in, left=1.0in, right=1.0in, bottom=1.0in]{geometry}

\usepackage{tikz}\tikzset{x=1cm,y=1cm,z=1cm}

\usepackage{pgfplots}\pgfplotsset{compat=1.16}

\usepackage[hyphens]{url}

\usepackage{hyperref}
\usepackage{breakurl}

\title{Algebraic solution of project scheduling problems with temporal constraints\thanks{Operational Research 24(4), 69 (2024). https://link.springer.com/article/10.1007/s12351-024-00880-3}}

\author{N. Krivulin\thanks{Faculty of Mathematics and Mechanics, St.~Petersburg State University, Universitetsky Ave.~28, St.~Petersburg, 198504, Russia; 
nkk@math.spbu.ru.}
\and
S. Gubanov\thanks{St.~Petersburg Branch, JSC ``Design Bureau ``Lutch'', Akademika Pavlova Str.~14a, St.~Petersburg, 197376, Russia; segubanov@mail.ru}
}

\date{}

\newtheorem{theorem}{Theorem}
\newtheorem{lemma}[theorem]{lemma}

\theoremstyle{definition}

\newtheorem{example}{Example}

\begin{document}

\maketitle

\begin{abstract}
New solutions for problems in optimal scheduling of activities in a project under temporal constraints are developed in the framework of tropical algebra which deals with the theory and application of algebraic systems with idempotent operations. We start with a constrained tropical optimization problem that has an objective function represented as a vector form given by an arbitrary matrix, and that can be solved analytically in a closed but somewhat complicated form. We examine a special case of the problem when the objective function is given by a matrix of unit rank, and show that the solution can be sufficiently refined in this case, which results in an essentially simplified analytical form and reduced computational complexity of the solution. We exploit the obtained result to find complete solutions of project scheduling problems to minimize the project makespan and the maximum absolute deviation of start times of activities under temporal constraints. The constraints under consideration include ``start--start'', ``start--finish'' and ``finish--start'' precedence relations, release times, release deadlines and completion deadlines for activities. As an application, we consider optimal scheduling problems of a vaccination project in a medical centre.
\\

\textbf{Key-Words:} idempotent semiefield, tropical optimization, minimax optimization problem, temporal project scheduling, project management.
\\

\textbf{MSC (2020):} 90C24, 15A80, 90C47, 90B35
\end{abstract}

\section{Introduction}

Project scheduling, which is concerned with the optimal assignment of start and finish times to the activities in a project subject to temporal and resource constraints, forms an integral part of project management. In general, the project scheduling problems with various scarce resources and complex objectives can be rather difficult to solve. To handle these problems, many of which are known to be $NP$-complete, a range of solution techniques is used, including methods and algorithms of integer and mixed integer linear programming, combinatorial and graph optimization, as well as a variety of heuristic optimization techniques \cite{Demeulemeester2002Project,Tkindt2006Multicriteria,Vanhoucke2012Project,Ballestin2015Handbook}.

Scheduling problems that have temporal objectives and constraints, and do not take into account other scarce resources (manpower, machines, tools, equipment, space, money), constitute an important  class of temporal project scheduling problems \cite{Demeulemeester2002Project,Neumann2003Project}. These problems may occur as auxiliary problems in the general project scheduling and are of independent interest. 

The typical temporal objective functions to be minimized (or sometimes maximized, depending on the criterion used) include the makespan (duration) of the project, the absolute deviation of the finish times of activities from due dates and the absolute deviation (variation) of the start times of activities. The temporal constraints appear in the form of various precedence relationships (``start--start'', ``start--finish'', ``finish--start'') between the start and finish times of activities, and box constraints (release time, release deadline, completion deadline) on the start and finish times.

In many cases, temporal scheduling problems can be formulated as minimax optimization problems and represented as linear programs, which are solved using computational methods and algorithms of linear programming. This approach offers iterative numerical procedures to obtain a solution if it exists, but does not provide all solutions in a direct analytical form.

Another approach, which shows the potential to solve many temporal scheduling problems analytically, is to apply models and methods of tropical (idempotent) algebra dealing with the theory and application of algebraic systems with idempotent operations \cite{Golan2003Semirings,Heidergott2006Maxplus,Gondran2008Graphs,Butkovic2010Maxlinear,Maclagan2015Introduction}.
 
Since the early research works on tropical algebra in the 1960s \cite{Pandit1961New,Cuninghamegreen1962Describing,Giffler1963Scheduling}, temporal scheduling problems have been applied to motivate and illustrate the development. In subsequent years, these problems continued to serve as important applications and to inspire further researches in the area (see, e.g., \cite{Zimmermann2003Disjunctive,Bouquard2006Application,Goto2009Robust,Singh2014Efficient,Goto2017Forward}).

An efficient technique to handle temporal scheduling problems is proposed and developed in \cite{Krivulin2015Extremal,Krivulin2017Direct,Krivulin2017Tropical,Krivulin2017Tropicaloptimization,Krivulin2020Tropical}, based on the formulation and solution of the scheduling problems as optimization problems in the tropical algebra setting (tropical optimization problems). The analytical solution of the project scheduling problems may be rather cumbersome and involve many matrix-vector operations. However, as it is shown in \cite{Krivulin2021Algebraicsolution} (see also some examples in \cite{Krivulin2017Tropical,Krivulin2017Tropicaloptimization}), the solution of the problems with specific objectives, such as minimizing the absolute deviation of start times of activities, can be simplified.

In this paper, new solutions for problems of optimal scheduling of activities in a project under temporal constraints are developed in the framework of tropical algebra. We start with a constrained tropical optimization problem that has an objective function represented as a vector form given by an arbitrary matrix, and that can be solved analytically in a closed but somewhat complicated form. We examine a special case of the problem when the objective function is given by a matrix of unit rank, and show that the solution can be sufficiently refined in this case, which results in an essentially simplified analytical form and reduced computational complexity of the solution.

Next, we exploit the obtained result to find complete analytical solutions of project scheduling problems to minimize the project makespan and the maximum absolute deviation of start times of activities under temporal constraints. The constraints under consideration include ``start--start'', ``start--finish'' and ``finish--start'' precedence relations, release times, release deadlines and completion deadlines for activities. As an application, we consider optimal scheduling problems of a vaccination project in a medical centre in the healthcare industry.

The purpose of this study is to develop a general solution approach based on a tropical optimization technique proposed in \cite{Krivulin2021Algebraicsolution} to handle a class of optimization problems that are frequently encountered in project scheduling. The technique uses a special form of the objective function in the problems, which allows one to simplify the formal representation and reduce the computational complexity of solutions. As an application example, a temporal scheduling problem to minimize the maximum absolute deviation of start times of activities is solved in \cite{Krivulin2021Algebraicsolution} under few constraints on start and finish times of activities. To extend and generalize this result, we now offer an analytical solution to a tropical optimization problem with a general objective function and an extended system of constraints. We apply this solution to new temporal project scheduling problems (including a problem of common interest to minimize the project makespan), which results in new complete solutions of the problems in a compact vector form.

The contribution of the study is twofold. First, a new approach is developed to solve a range of temporal project scheduling problems analytically in an explicit form suitable for further analysis with formal methods and for numerical computations with moderate complexity. The proposed analytical solution has a strong potential to complement and supplement the existing approaches, which are based on numerical algorithms. Second, new solutions are obtained for a class of tropical optimization problems, which can find application to solve actual problems in operations research, management science and other fields. This result has implications for further development of methods and extension of applications of tropical optimization.

The rest of the paper is organized as follows. In Section~\ref{S-PSP}, we describe the temporal project scheduling problems of interest and represent them in the usual algebraic form. In Section~\ref{S-TA}, we outline the key definitions, notations and results of tropical algebra to provide a formal framework for further discussion. Section~\ref{S-TOP} includes our main result, which provides a new solution to an optimization problem given in the tropical algebra setting. The result obtained is applied in Section~\ref{S-SPSP} to derive solutions of the scheduling problems formulated above. In Section~\ref{S-AVS}, we discuss the application of the solutions to the optimal scheduling of vaccination sessions and give illustrative numerical examples. Section~\ref{S-C} presents some conclusion remarks.

\section{Project scheduling problems}
\label{S-PSP}

We consider project scheduling problems with temporal (time-based) objectives and constraints, which may typically occur in the project management practice in various areas \cite{Demeulemeester2002Project,Neumann2003Project,Tkindt2006Multicriteria,Vanhoucke2012Project,Ballestin2015Handbook}.

Suppose that a project consists of a set of activities that have to be performed in parallel to complete the project. The start and finish times of each activity are subject to constraints that are derived from technical, managerial, economic and other limitations and conditions involved in the implementation of the project. These temporal constraints include various forms of precedence relationships between and box constraints on the start and finish times of activities, which allow to represent a wide range of real-life situations.

The scheduling problems under consideration are to determine the start and finish times for each activity to find optimal schedules that achieve specified objectives under the given temporal constraints. One of the objectives is set to minimize the project makespan (total duration), which constitutes a natural optimality criterion in many scheduling problems. Another objective is to minimize the absolute deviation of start times over all activities, and it is of interest when for some technological, organizational or other reasons, all activities in the project have to be scheduled to start as simultaneously as possible.

Consider a project that involves $n$ activities, and introduce the following notation. For each activity $i=1,\ldots,n$, we denote by $x_{i}$ the start time and by $y_{i}$ the finish time to be scheduled. Let $b_{ij}$ be a given minimal allowed time lag between the start of activity $j$ and start of $i$, which specifies the ``start--start'' precedence relationship having the form of the inequality $b_{ij}+x_{j}\leq x_{i}$. If this lag is undefined, it is set to $-\infty$. Combining such inequalities for all $j$ yields one equivalent inequality
\begin{equation}
\max_{1\leq j\leq n}
(b_{ij}+x_{j})
\leq
x_{i}.
\label{I-maxbijxj-xi}
\end{equation}

Suppose that a minimal allowed time lag $c_{ij}$ between the start of activity $j$ and finish of $i$ is given or defined as $c_{ij}=-\infty$ otherwise. The ``start--finish'' relationship is represented by the inequality $c_{ij}+x_{j}\leq y_{i}$. We replace the inequalities for all $j$ by one combined inequality  
\begin{equation*}
\max_{1\leq j\leq n}
(c_{ij}+x_{j})
\leq
y_{i}.
\end{equation*}

Each activity $i$ is assumed to finish immediately once all ``start--finish'' constraints are satisfied, which makes at least one inequality hold as an equality. As a result, we further replace the combined inequality by the equality
\begin{equation}
\max_{1\leq j\leq n}
(c_{ij}+x_{j})
=
y_{i}.
\label{E-maxcijxj-xi}
\end{equation}

Furthermore, given a minimal allowed time lag $d_{ij}$ between the finish of activity $j$ and start of $i$ ($d_{ij}=-\infty$ if undefined), the ``finish--start'' precedence relationship is the inequality $d_{ij}+y_{j}\leq x_{i}$, which yields the inequality 
\begin{equation}
\max_{1\leq j\leq n}
(d_{ij}+y_{j})
\leq
x_{i}.
\label{I-maxdijxj-xi}
\end{equation}

Finally, let $g_{i}$ be a given release time and $h_{i}$ release deadline, which indicate the earliest and latest allowed start time, and $f_{i}$ be a given finish deadline, which indicates the latest allowed finish time. The release time and deadline constraints are represented as box constraints by the inequalities 
\begin{equation}
g_{i}
\leq
x_{i}
\leq
h_{i},
\qquad
y_{i}
\leq
f_{i}.
\label{I-gixihi-yifi}
\end{equation}

As an optimality criterion for scheduling, we consider the project makespan defined as the time interval between the earliest start time and the latest finish time of activities. This criterion, which is to be minimized, is defined as
\begin{equation}
\max_{1\leq i\leq n}y_{i}
-
\min_{1\leq i\leq n}x_{i}
=
\max_{1\leq i\leq n}y_{i}
+
\max_{1\leq i\leq n}(-x_{i}).
\label{E-maxyiminxi}
\end{equation}

Another criterion to be minimized is the maximum absolute deviation of start times of activities, which is written in the form 
\begin{equation}
\max_{1\leq i\leq n}x_{i}
-
\min_{1\leq i\leq n}x_{i}
=
\max_{1\leq i\leq n}x_{i}
+
\max_{1\leq i\leq n}(-x_{i}).
\label{E-maxximinxi}
\end{equation}

We combine the constraints at \eqref{I-maxbijxj-xi}--\eqref{I-gixihi-yifi} with the objective at \eqref{E-maxyiminxi} to state the following problem of minimizing the project makespan. For given parameters $b_{ij}$, $c_{ij}$, $d_{ij}$, $g_{i}$, $h_{i}$ and $f_{i}$, where $i,j=1,\ldots,n$, find $x_{i}$ and $y_{i}$ that achieve
\begin{equation}
\begin{aligned}
&
\min_{x_{i},y_{i}}
&&
\max_{1\leq i\leq n}y_{i}
+
\max_{1\leq i\leq n}(-x_{i});
\\
&
&&
\max_{1\leq j\leq n}
(b_{ij}+x_{j})
\leq
x_{i},
\quad
\max_{1\leq j\leq n}
(c_{ij}+x_{j})
=
y_{i},
\\
&
&&
\max_{1\leq j\leq n}
(d_{ij}+y_{j})
\leq
x_{i},
\quad
g_{i}
\leq
x_{i}
\leq
h_{i},
\quad
y_{i}
\leq
f_{i},
\quad
i=1,\ldots,n.
\end{aligned}
\label{P-minmaxyimaxxi}
\end{equation}

The replacement of the objective function by the function at \eqref{E-maxximinxi} leads to a problem of minimizing the maximum deviation of start times in the form
\begin{equation}
\begin{aligned}
&
\min_{x_{i},y_{i}}
&&
\max_{1\leq i\leq n}x_{i}
+
\max_{1\leq i\leq n}(-x_{i});
\\
&
&&
\max_{1\leq j\leq n}
(b_{ij}+x_{j})
\leq
x_{i},
\quad
\max_{1\leq j\leq n}
(c_{ij}+x_{j})
=
y_{i},
\\
&
&&
\max_{1\leq j\leq n}
(d_{ij}+y_{j})
\leq
x_{i},
\quad
g_{i}
\leq
x_{i}
\leq
h_{i},
\quad
y_{i}
\leq
f_{i},
\quad
i=1,\ldots,n.
\end{aligned}
\label{P-minmaxximaxxi}
\end{equation}

We observe that minimax problems \eqref{P-minmaxyimaxxi} and \eqref{P-minmaxximaxxi} can be readily rewritten as linear programs and then solved using computational methods and algorithms available in mathematical programming, such as the classical simplex method and Karmarkar algorithm \cite{Karmarkar1984New}. This approach offers iterative procedures that provide an approximate numerical solution if the problem is solvable, or indicates the nonexistence of solutions otherwise. In what follows, we show how to handle the problems analytically using methods and results of tropical mathematics to derive a direct representation of all solutions in a compact vector form.

\section{Tropical algebra}
\label{S-TA}

We begin with an overview of key definitions, notations and results of tropical (idempotent) algebra, which provide a formal framework for the study. For more details on the theory and application of tropical algebra, one can consult the monographs and textbooks by \cite{Golan2003Semirings,Heidergott2006Maxplus,Gondran2008Graphs,Butkovic2010Maxlinear,Maclagan2015Introduction}.

\subsection{Idempotent semifield}

Consider a set $\mathbb{X}$ that contains two distinct elements zero $\mathbb{0}$ and one $\mathbb{1}$, and is closed under addition $\oplus$ and multiplication $\otimes$. We assume that the triplet $(\mathbb{X},\mathbb{0},\oplus)$ is a commutative idempotent monoid, $(\mathbb{X}\setminus\{\mathbb{0}\},\mathbb{1},\otimes)$ is an Abelian group, and multiplication is distributive over addition. The algebraic system $(\mathbb{X},\mathbb{0},\mathbb{1},\oplus,\otimes)$ is commonly referred to as the tropical (idempotent) semifield.

In the semifield, addition is idempotent, which means that $x\oplus x=x$ for any $x\in\mathbb{X}$, and does not allow inverses. Addition determines a partial order such that $x\leq y$ for some $x,y\in\mathbb{X}$ if and only if $x\oplus y=y$. Both addition and multiplication are isotone in each argument: the inequality $x\leq y$ implies the inequalities $x\oplus z\leq y\oplus z$ and $xz\leq yz$ for any $z$ (the multiplication sign $\otimes$ is omitted here and hereafter to save writing).

Furthermore, with respect to the partial order, addition possesses the extremal property (the majority law) that both inequalities $x\leq x\oplus y$ and $y\leq x\oplus y$ are valid for any $x,y$. The inequality $x\oplus y\leq z$ is equivalent to the system of inequalities $x\leq z$ and $y\leq z$. We assume that the partial order extends to a total order to make the idempotent semifield be linearly ordered. 

Every nonzero $x$ has an inverse $x^{-1}$ such that $xx^{-1}=\mathbb{1}$. Taking inverse is antitone: for nonzero $x,y\in\mathbb{X}$, the inequality $x\leq y$ results in $x^{-1}\geq y^{-1}$. Integer exponents routinely represent iterative multiplication as $x^{p}=xx^{p-1}$, $x^{-p}=(x^{-1})^{p}$ and $\mathbb{0}^{p}=\mathbb{0}$ for any $x\ne\mathbb{0}$ and integer $p>0$. We assume that the equation $x^{p}=a$ has a unique solution in $x$ for any $a\in\mathbb{X}$ and integer $p>0$ to make the semifield radicable, where the rational exponents are defined as well. 

For any $a,b\in\mathbb{X}$ and rational $q>0$, the binomial identity takes the form $(a\oplus b)^{q}=a^{q}\oplus b^{q}$, whereas the inequality between the arithmetic and geometric means becomes $a\oplus b\geq (ab)^{1/2}$. The extensions of both identity and inequality to more than two summands are straightforward.

An example of a linearly ordered, radicable, idempotent semifield is the real system $\mathbb{R}_{\max,+}=(\mathbb{R}\cup\{-\infty\},-\infty,0,\max,+)$, also known as max-plus algebra. In this semifield, addition is defined as maximum and multiplication as arithmetic addition with the zero $\mathbb{0}$ and one $\mathbb{1}$ given respectively by $-\infty$ and $0$. The inverse for $x\ne\mathbb{0}$ corresponds to the conventional opposite number $-x$, the exponentiation $x^{y}$ coincides with the standard multiplication $xy$, and the order determined by addition is equivalent to the natural linear order on $\mathbb{R}$.

Another example is min-algebra $\mathbb{R}_{\min}=(\mathbb{R}_{+}\cup\{+\infty\},+\infty,1,\min,\times)$, where $\mathbb{R}_{+}$ is the set of positive reals, $\oplus=\min$, $\otimes=\times$, $\mathbb{0}=+\infty$ and $\mathbb{1}=1$. The operations of inversion and exponentiation are defined as usual, whereas the order induced by addition is opposite to the natural linear order on $\mathbb{R}_{+}$.

\subsection{Algebra of matrices and vectors}

Consider matrices and vectors of elements from $\mathbb{X}$ and denote the set of matrices with $m$ rows and $n$ columns by $\mathbb{X}^{m\times n}$, and the set of column vectors with $n$ elements by $\mathbb{X}^{n}$. A matrix (vector) with all entries equal to zero $\mathbb{0}$ is the zero matrix (vector) denoted by $\bm{0}$. If a vector has no zero entries, it is called regular. A matrix without zero columns is called column-regular.  

Addition, multiplication and scalar multiplication of matrices and vectors are performed according to the standard entrywise rules, where the arithmetic addition and multiplication are replaced by $\oplus$ and $\otimes$. The monotonicity of the scalar operations $\oplus$ and $\otimes$ extends to matrix and vector operations, where the inequalities are interpreted entrywise.

For a vector $\bm{x}$, its transpose is denoted $\bm{x}^{T}$. The multiplicative conjugate transpose of a nonzero column vector $\bm{x}=(x_{i})$ is the row vector $\bm{x}^{-}=(x_{i}^{-})$, where $x_{i}^{-}=x_{i}^{-1}$ if $x_{i}\ne\mathbb{0}$, and $x_{i}^{-}=\mathbb{0}$ otherwise.

Let $\bm{A}=(a_{ij})$ be a square matrix in $\mathbb{X}^{n\times n}$. The trace of $\bm{A}$ is calculated as
\begin{equation*}
\mathop\mathrm{tr}\bm{A}
=
a_{11}\oplus\cdots\oplus a_{nn}
=
\bigoplus_{k=1}^{n}a_{kk}.
\end{equation*}

For conforming matrices $\bm{A}$, $\bm{B}$, $\bm{C}$ and scalar $x$, the following equalities hold:
\begin{equation*}
\mathop\mathrm{tr}(\bm{A}\oplus\bm{B})
=
\mathop\mathrm{tr}\bm{A}
\oplus
\mathop\mathrm{tr}\bm{B},
\qquad
\mathop\mathrm{tr}(\bm{A}\bm{C})
=
\mathop\mathrm{tr}(\bm{C}\bm{A}),
\qquad
\mathop\mathrm{tr}(x\bm{A})
=
x\mathop\mathrm{tr}\bm{A}.
\end{equation*}

A square matrix with all entries equal to $\mathbb{1}$ on the diagonal and to $\mathbb{0}$ elsewhere, is the identity matrix denoted $\bm{I}$. The nonnegative integer exponent of a matrix $\bm{A}\in\mathbb{X}^{n\times n}$ is defined for any integer $p>0$ as $\bm{A}^{p}=\bm{A}^{p-1}\bm{A}$ and $\bm{A}^{0}=\bm{I}$.

For any matrix $\bm{A}\in\mathbb{X}^{n\times n}$, we introduce the trace function
\begin{equation*}
\mathop\mathrm{Tr}(\bm{A})
=
\mathop\mathrm{tr}\bm{A}
\oplus
\cdots
\oplus
\mathop\mathrm{tr}\bm{A}^{n}
=
\bigoplus_{k=1}^{n}\mathop\mathrm{tr}\bm{A}^{k}.
\end{equation*}

If $\mathop\mathrm{Tr}(\bm{A})\leq\mathbb{1}$, the Kleene star operator maps $\bm{A}$ to the matrix 
\begin{equation*}
\bm{A}^{\ast}
=
\bm{I}
\oplus
\bm{A}
\oplus
\cdots
\oplus
\bm{A}^{n-1}
=
\bigoplus_{k=0}^{n-1}\bm{A}^{k}.
\end{equation*}

The tropical spectral radius of the matrix $\bm{A}$ is given by
\begin{equation*}
\rho(\bm{A})
=
\mathop\mathrm{tr}\bm{A}
\oplus
\cdots
\oplus
\mathop\mathrm{tr}\nolimits^{1/n}(\bm{A}^{n})
=
\bigoplus_{k=1}^{n}\mathop\mathrm{tr}\nolimits^{1/k}(\bm{A}^{k}).
\end{equation*}

Let $\bm{1}=(\mathbb{1},\ldots,\mathbb{1})^{T}$ denote the vector of ones. For any vector $\bm{x}$ and matrix $\bm{A}$, tropical analogues of vector and matrix norms are respectively defined as
\begin{equation*}
\|\bm{x}\|
=
\bm{1}^{T}\bm{x}
=
\bm{x}^{T}\bm{1},
\qquad
\|\bm{A}\|
=
\bm{1}^{T}\bm{A}\bm{1}.
\end{equation*}

We conclude this overview with a solution of a vector inequality. Suppose that given a matrix $\bm{A}\in\mathbb{X}^{m\times n}$ and vector $\bm{b}\in\mathbb{X}^{m}$, the problem is to find vectors $\bm{x}\in\mathbb{X}^{n}$ that satisfy the inequality
\begin{equation}
\bm{A}\bm{x}
\leq
\bm{b}.
\label{I-Axb}
\end{equation}

A complete solution of this problem can be given as follows.
\begin{lemma}
\label{L-Axb}
For any column-regular matrix $\bm{A}$ and regular vector $\bm{b}$, the vector $\bm{x}$ is a solution to inequality \eqref{I-Axb} if and only if
\begin{equation*}
\bm{x}
\leq
(\bm{b}^{-}\bm{A})^{-}.
\end{equation*}
\end{lemma}

\section{Tropical optimization problems}
\label{S-TOP}

In this section, we focus on optimization problems, which are formulated and solved in the tropical algebra setting, and hence referred to as tropical optimization problems. We start with a problem of constrained minimization of a multiplicative conjugate quadratic (pseudoquadratic) form $\bm{x}^{-}\bm{A}\bm{x}$, where $\bm{A}$ is a given square matrix and $\bm{x}$ is an unknown column vector. We describe an existing solution of the problem, known in a rather cumbersome matrix-vector form, and make some comments on the solution.

Next, we consider a special case of the problem, where the above quadratic form is defined using a rank-one matrix. As a key result, we refine the existing solution to take into account the special form of the matrix, which leads to results in a more compact and transparent form.

\subsection{Minimization of conjugate quadratic forms}

We now consider a problem to minimize a conjugate quadratic form subject to linear and box constraints on the solution vector. Given matrices $\bm{A},\bm{B}\in\mathbb{X}^{n\times n}$ and vectors $\bm{g},\bm{h}\in\mathbb{X}^{n}$, the problem is to find regular vectors $\bm{x}\in\mathbb{X}^{n}$ that attain the minimum
\begin{equation}
\begin{aligned}
\min_{\bm{x}>\bm{0}}
&&&
\bm{x}^{-}\bm{A}\bm{x};
\\
\text{s.t.}
&&&
\bm{B}\bm{x}
\leq
\bm{x},
\quad
\bm{g}
\leq
\bm{x}
\leq
\bm{h}.
\label{P-minxAx-Bxx-gxh}
\end{aligned}
\end{equation}

A complete solution of problem \eqref{P-minxAx-Bxx-gxh} can be obtained for instance as a direct consequence of the solution of a problem with an extended set of constraints or with a more general objective function in \cite{Krivulin2017Direct}. In what follows, we use the solution given in the next claim.
\begin{theorem}
\label{T-minxAx-Bxx-gxh}
Let $\bm{A}$ be a matrix with $\rho(\bm{A})>\mathbb{0}$, and $\bm{B}$ matrix for which $\mathop\mathrm{Tr}(\bm{B})\leq\mathbb{1}$. Let $\bm{g}$ be a vector and $\bm{h}$ regular vector such that $\bm{h}^{-}\bm{B}^{\ast}\bm{g}\leq\mathbb {1}$. Denote
\begin{gather}
\bm{F}_{k}
=
\bigoplus_{0\leq i_{1}+\cdots+i_{k}\leq n-k}\bm{A}\bm{B}^{i_{1}}\cdots\bm{A}\bm{B}^{i_{k}},
\qquad
k=1,\ldots,n;
\label{E-Fk}
\\
\bm{G}_{k}
=
\bigoplus_{0\leq i_{0}+i_{1}+\cdots+i_{k}\leq n-k-1}
\bm{B}^{i_{0}}(\bm{A}\bm{B}^{i_{1}}\cdots\bm{A}\bm{B}^{i_{k}}),
\qquad
k=1,\ldots,n-1.
\label{E-Gk}
\end{gather}

Then, the minimum value of the objective function in problem \eqref{P-minxAx-Bxx-gxh} is equal to
\begin{equation}
\theta
=
\bigoplus_{k=1}^{n}\mathop\mathrm{tr}\nolimits^{1/k}(\bm{F}_{k})
\oplus
\bigoplus_{k=1}^{n-1}(\bm{h}^{-}\bm{G}_{k}\bm{g})^{1/k},
\label{E-theta}
\end{equation}
and all regular solutions of the problem are given in the parametric form
\begin{equation}
\bm{x}
=
\bm{G}\bm{u},
\qquad
\bm{G}
=
(\theta^{-1}\bm{A}\oplus\bm{B})^{\ast}
=
\bigoplus_{k=1}^{n-1}
\theta^{-k}
\bm{G}_{k}
\oplus
\bm{B}^{\ast},
\label{E-x}
\end{equation}
where $\bm{u}\ne\bm{0}$ is a vector of parameters that satisfies the condition
\begin{equation}
\bm{g}
\leq
\bm{u}
\leq
(\bm{h}^{-}\bm{G})^{-}.
\label{I-u}
\end{equation}
\end{theorem}

In the statement of the theorem, the condition $\mathop\mathrm{Tr}(\bm{B})\leq\mathbb{1}$ is responsible for the existence of regular solutions of the linear inequality $\bm{B}\bm{x}\leq\bm{x}$, whereas the condition $\bm{h}^{-}\bm{B}^{\ast}\bm{g}\leq\mathbb{1}$ ensures that both the linear inequality constraint and the box constraint $\bm{g}\leq\bm{x}\leq\bm{h}$ have a common regular solution.

Note that the solution provided by the theorem has a computational complexity of order not exceeding $O(n^{5})$, where the most computationally expensive part which yields the upper bound is the evaluation of the minimum $\theta$ (see, e.g., \cite{Krivulin2017Direct,Krivulin2017Tropicaloptimization,Krivulin2022Minimizing}).  

Consider a special case of problem \eqref{P-minxAx-Bxx-gxh}, where the matrix $\bm{A}$ has proportional (collinear) columns. Given vectors $\bm{p},\bm{q}\in\mathbb{X}^{n}$, $\bm{q}\ne\bm{0}$, suppose that the matrix is defined as $\bm{A}=\bm{p}\bm{q}^{-}$ and thus has unit rank. We now formulate a problem of minimizing the conjugate quadratic form with rank-one matrix as
\begin{equation}
\begin{aligned}
&
\min_{\bm{x}>\bm{0}}
&&
\bm{x}^{-}\bm{p}\bm{q}^{-}\bm{x};
\\
&&&
\bm{B}\bm{x}
\leq
\bm{x},
\quad
\bm{g}
\leq
\bm{x}
\leq
\bm{h}.
\end{aligned}
\label{P-minxpqx-Bxx-gxh}
\end{equation}

Several versions of problem \eqref{P-minxpqx-Bxx-gxh} with either a reduced set of constraints or less general objective functions are examined in \cite{Krivulin2017Tropical,Krivulin2017Tropicaloptimization,Krivulin2020Tropical,Krivulin2021Algebraicsolution} using somewhat different techniques. Below we give a solution to the problem set in the general form at \eqref{P-minxpqx-Bxx-gxh}, which combines and summarizes results for the reduced versions indicated above.

\subsection{Conjugate quadratic form with rank-one matrix}

To solve problem \eqref{P-minxpqx-Bxx-gxh}, we apply the result of Theorem~\ref{T-minxAx-Bxx-gxh}, where we exploit the particular form of the matrix $\bm{A}$ to reduce the solution expressions as much as possible. We prove the following assertion.
\begin{theorem}
\label{T-minxpqx-Bxx-gxh}
Let $\bm{B}$ be a matrix for which $\mathop\mathrm{Tr}(\bm{B})\leq\mathbb{1}$. Let $\bm{g}$ be a vector and $\bm{h}$ regular vector such that $\bm{h}^{-}\bm{B}^{\ast}\bm{g}\leq\mathbb{1}$. Let $\bm{p}$ and $\bm{q}$ be nonzero vectors such that $\bm{q}^{-}\bm{p}>\mathbb{0}$.

Then, the minimum value of the objective function in problem \eqref{P-minxpqx-Bxx-gxh} is equal to
\begin{equation}
\theta
=
\bigoplus_{0\leq i+j\leq n-2}
(\bm{h}^{-}\bm{B}^{i}\bm{p})
(\bm{q}^{-}\bm{B}^{j}\bm{g})
\oplus
\bm{q}^{-}\bm{B}^{\ast}\bm{p},
\label{E-theta1}
\end{equation}
and all regular solutions are given in the parametric form
\begin{equation}
\bm{x}
=
\bm{G}
\bm{u},
\qquad
\bm{G}
=
\bigoplus_{0\leq i+j\leq n-2}
\theta^{-1}
\bm{B}^{i}\bm{p}
\bm{q}^{-}\bm{B}^{j}
\oplus
\bm{B}^{\ast},
\label{E-x1}
\end{equation}
where $\bm{u}\ne\bm{0}$ is a vector of parameters that satisfies the condition
\begin{equation}
\bm{g}
\leq
\bm{u}
\leq
(\bm{h}^{-}\bm{G})^{-}.
\label{I-u1}
\end{equation}
\end{theorem}
\begin{proof}
We start with the observation that under the condition $\bm{q}^{-}\bm{p}>\mathbb{0}$, the spectral radius of the matrix $\bm{p}\bm{q}^{-}$ is greater than zero as
\begin{equation*}
\rho(\bm{p}\bm{q}^{-})
=
\bigoplus_{k=1}^{n}\mathop\mathrm{tr}\nolimits^{1/k}(\bm{p}\bm{q}^{-})^{k}
=
\bm{q}^{-}\bm{p}
>
\mathbb{0}.
\end{equation*}

Since the assumptions of Theorem~\ref{T-minxAx-Bxx-gxh} with $\bm{A}=\bm{p}\bm{q}^{-}$ are fulfilled, we apply this theorem and then examine the expressions for the minimum $\theta$ of the objective function at \eqref{E-theta} and for the set of solution vectors $\bm{x}$ at \eqref{E-x} and \eqref{I-u}. To handle the minimum, we first simplify the formulas for the matrix sums $\bm{F}_{k}$ and $\bm{G}_{k}$ at \eqref{E-Fk} and \eqref{E-Gk}, and then evaluate both terms on the right-hand side of \eqref{E-theta}.

By substitution $\bm{A}=\bm{p}\bm{q}^{-}$ in the matrix product $\bm{A}\bm{B}^{i_{1}}\cdots\bm{A}\bm{B}^{i_{k}}$, we represent this product as
\begin{equation*}
\bm{p}\bm{q}^{-}\bm{B}^{i_{1}}\cdots\bm{p}\bm{q}^{-}\bm{B}^{i_{k}}
=
(\bm{q}^{-}\bm{B}^{i_{1}}\bm{p}\cdots\bm{q}^{-}\bm{B}^{i_{k-1}}\bm{p})\bm{p}\bm{q}^{-}\bm{B}^{i_{k}}.
\end{equation*}

In this case, the right-hand side of \eqref{E-Fk} becomes
\begin{equation}
\bm{F}_{k}
=
\bigoplus_{0\leq i_{1}+\cdots+i_{k}\leq n-k}
(\bm{q}^{-}\bm{B}^{i_{1}}\bm{p}
\cdots
\bm{q}^{-} \bm{B}^{i_{k-1}}\bm{p})
\bm{p}\bm{q}^{-}\bm{B}^{i_{k}}.
\label{E-Fk1}
\end{equation}

In the same way, we rewrite the product $\bm{B}^{i_{0}}(\bm{A}\bm{B}^{i_{1}}\cdots\bm{A}\bm{B}^{i_{k}})$ in the form
\begin{equation*}
\bm{B}^{i_{0}}(\bm{p}\bm{q}^{-}\bm{B}^{i_{1}}\cdots\bm{p}\bm{q}^{-}\bm{B}^{i_{k}})
=
(\bm{q}^{-}\bm{B}^{i_{1}}\bm{p}\cdots\bm{q}^{-}\bm{B}^{i_{k-1}}\bm{p})\bm{B}^{i_{0}}\bm{p}\bm{q}^{-}\bm{B}^{i_{k}},
\end{equation*}
and then obtain
\begin{equation}
\bm{G}_{k}
=
\bigoplus_{0\leq i_{0}+i_{1}+\cdots+i_{k}\leq n-k-1}
(\bm{q}^{-}\bm{B}^{i_{1}}\bm{p}\cdots\bm{q}^{-}\bm{B}^{i_{k-1}}\bm{p})\bm{B}^{i_{0}}\bm{p}\bm{q}^{-}\bm{B}^{i_{k}}.
\label{E-Gk1}
\end{equation}

To derive the first term in \eqref{E-theta}, we take trace of both sides of \eqref{E-Fk1} and use properties of trace, which results in
\begin{equation*}
\mathop\mathrm{tr}\bm{F}_{k}
=
\bigoplus_{\substack{0\leq i_{1}+\cdots+i_{k}\leq n-k\\i_{1},\ldots,i_{k}\geq0}}
\bm{q}^{-}\bm{B}^{i_{1}}\bm{p}\cdots\bm{q}^{-}\bm{B}^{i_{k}}\bm{p}.
\end{equation*}

We consider the sum of the roots of traces over all $k=1,\ldots,n$ and rearrange it by separating the summand for $k=1$ to write
\begin{equation*}
\bigoplus_{k=1}^{n}
\mathop\mathrm{tr}\nolimits^{1/k}(\bm{F}_{k})
=
\mathop\mathrm{tr}\bm{F}_{1}
\oplus
\bigoplus_{k=2}^{n}
\mathop\mathrm{tr}\nolimits^{1/k}(\bm{F}_{k}),
\qquad
\mathop\mathrm{tr}\bm{F}_{1}
=
\bigoplus_{k=0}^{n-1}
\bm{q}^{-}\bm{B}^{k}\bm{p}
=
\bm{q}^{-}\bm{B}^{\ast}\bm{p}.
\end{equation*}

By applying the inequality between the arithmetic and geometric means and the extremal property of addition, we have
\begin{equation*}
(\bm{q}^{-}\bm{B}^{i_{1}}\bm{p}\cdots\bm{q}^{-}\bm{B}^{i_{k}}\bm{p})^{1/k}
\leq
\bm{q}^{-}\bm{B}^{i_{1}}\bm{p}
\oplus\cdots\oplus
\bm{q}^{-}\bm{B}^{i_{k}}\bm{p}
\leq
\bigoplus_{k=0}^{n-1}
\bm{q}^{-}\bm{B}^{k}\bm{p}
=
\bm{q}^{-}\bm{B}^{\ast}\bm{p}.
\end{equation*}

With this result, we see that the first summand dominates the others since
\begin{equation*}
\mathop\mathrm{tr}\nolimits^{1/k}(\bm{F}_{k})
=
\bigoplus_{\substack{0\leq i_{1}+\cdots+i_{k}\leq n-k\\i_{1},\ldots,i_{k}\geq0}}
(\bm{q}^{-}\bm{B}^{i_{1}}\bm{p}\cdots\bm{q}^{-}\bm{B}^{i_{k}}\bm{p})^{1/k}
\leq
\bm{q}^{-}\bm{B}^{\ast}\bm{p}
=
\mathop\mathrm{tr}\bm{F}_{1}.
\end{equation*}

In this case, it follows from the extremal property of addition that
\begin{equation}
\bigoplus_{k=1}^{n}
\mathop\mathrm{tr}\nolimits^{1/k}(\bm{F}_{k})
=
\mathop\mathrm{tr}\bm{F}_{1}.
\label{E-F1}
\end{equation}

We now consider the second term on the right-hand side of \eqref{E-theta}. Multiplication of both sides of \eqref{E-Gk1} by $\bm{h}^{-}$ from the left and by $\bm{g}$ from the right yields
\begin{equation*}
\bm{h}^{-}\bm{G}_{k}\bm{g}
=
\bigoplus_{0\leq i_{0}+i_{1}+\cdots+i_{k}\leq n-k-1}
(\bm{q}^{-}\bm{B}^{i_{1}}\bm{p}\cdots\bm{q}^{-}\bm{B}^{i_{k-1}}\bm{p})\bm{h}^{-}\bm{B}^{i_{0}}\bm{p}\bm{q}^{-}\bm{B}^{i_{k}}\bm{g}.
\end{equation*}

After taking roots and summing over all $k=1,\ldots,n-1$, we arrive at the sum
\begin{equation*}
\bigoplus_{k=1}^{n-1}
(\bm{h}^{-}\bm{G}_{k}\bm{g})^{1/k}
=
\bm{h}^{-}\bm{G}_{1}\bm{g}
\oplus
\bigoplus_{k=2}^{n-1}
(\bm{h}^{-}\bm{G}_{k}\bm{g})^{1/k},
\end{equation*}
where
\begin{equation*}
\bm{h}^{-}\bm{G}_{1}\bm{g}
=
\bigoplus_{0\leq i_{0}+i_{1}\leq n-2}
\bm{h}^{-}\bm{B}^{i_{0}}\bm{p}\bm{q}^{-}\bm{B}^{i_{1}}\bm{g}.
\end{equation*}

Next, we again apply the inequality between the arithmetic and geometric means, and use the extremal property of addition to write
\begin{multline*}
(
(\bm{q}^{-}\bm{B}^{i_{1}}\bm{p}\cdots\bm{q}^{-}\bm{B}^{i_{k-1}}\bm{p})
\bm{h}^{-}\bm{B}^{i_{0}}\bm{p}
\bm{q}^{-}\bm{B}^{i_{k}}\bm{g}
)^{1/k}
\\
\leq
(\bm{q}^{-}\bm{B}^{i_{1}}\bm{p}
\oplus\cdots\oplus
\bm{q}^{-}\bm{B}^{i_{k-1}}\bm{p})
\oplus
\bm{h}^{-}\bm{B}^{i_{0}}\bm{p}
\bm{q}^{-}\bm{B}^{i_{k}}\bm{g}
\leq
\mathop\mathrm{tr}\bm{F}_{1}
\oplus
\bm{h}^{-}\bm{G}_{1}\bm{g}.
\end{multline*}

We use the last inequality to derive an upper bound for the second terms as
\begin{multline*}
\bigoplus_{k=1}^{n-1}
(\bm{h}^{-}\bm{G}_{k}\bm{g})^{1/k}
\\
=
\bigoplus_{k=1}^{n-1}
\bigoplus_{0\leq i_{0}+i_{1}+\cdots+i_{k}\leq n-k-1}
((\bm{q}^{-}\bm{B}^{i_{1}}\bm{p}\cdots\bm{q}^{-}\bm{B}^{i_{k-1}}\bm{p})\bm{h}^{-}\bm{B}^{i_{0}}\bm{p}\bm{q}^{-}\bm{B}^{i_{k}}\bm{g})^{1/k}
\\
\leq
\mathop\mathrm{tr}\bm{F}_{1}
\oplus
\bm{h}^{-}\bm{G}_{1}\bm{g}.
\end{multline*}

By combining the obtained result with \eqref{E-F1}, we arrive at the double inequality
\begin{equation*}
\mathop\mathrm{tr}\bm{F}_{1}
\oplus
\bm{h}^{-}\bm{G}_{1}\bm{g}
\leq
\bigoplus_{k=1}^{n}
\mathop\mathrm{tr}\nolimits^{1/k}(\bm{F}_{k})
\oplus
\bigoplus_{k=1}^{n-1}
(\bm{h}^{-}\bm{G}_{k}\bm{g})^{1/k}
\leq
\mathop\mathrm{tr}\bm{F}_{1}
\oplus
\bm{h}^{-}\bm{G}_{1}\bm{g}.
\end{equation*}

It follows from this double inequality that 
\begin{equation*}
\theta
=
\mathop\mathrm{tr}\bm{F}_{1}
\oplus
\bm{h}^{-}\bm{G}_{1}\bm{g}
=
\bigoplus_{0\leq i+j\leq n-2}
\bm{h}^{-}\bm{B}^{i}\bm{p}\bm{q}^{-}\bm{B}^{j}\bm{g}
\oplus
\bm{q}^{-}\bm{B}^{\ast}\bm{p},
\end{equation*}
which leads to the representation of $\theta$ in the form of \eqref{E-theta1}.

It remains to examine the solution given by \eqref{E-x} and \eqref{I-u}, and rearrange the matrix $\bm{G}$, which generates all solutions. First, we rewrite the expression for $\bm{G}$ at \eqref{E-x} as follows:
\begin{equation*}
\bm{G}
=
(\theta^{-1}\bm{p}\bm{q}^{-}\oplus\bm{B})^{\ast}
=
\theta^{-1}
\bm{G}_{1}
\oplus
\bigoplus_{k=2}^{n-1}
\theta^{-k}
\bm{G}_{k}
\oplus
\bm{B}^{\ast},
\end{equation*}
where
\begin{equation*}
\bm{G}_{1}
=
\bigoplus_{0\leq i_{0}+i_{1}\leq n-2}
\bm{B}^{i_{0}}\bm{p}\bm{q}^{-}\bm{B}^{i_{1}}.
\end{equation*}

We note that $\theta\geq\bm{q}^{-}\bm{B}^{\ast}\bm{p}\geq\bm{q}^{-}\bm{B}^{i}\bm{p}$, which yields the inequality $\theta^{-1}\bm{q}^{-}\bm{B}^{i}\bm{p}\leq\mathbb{1}$ for all $i=0,\ldots,n-1$. With this inequality, we have
\begin{equation*}
\theta^{-k}
(\bm{q}^{-}\bm{B}^{i_{1}}\bm{p}\cdots\bm{q}^{-}\bm{B}^{i_{k-1}}\bm{p})\bm{B}^{i_{0}}\bm{p}\bm{q}^{-}\bm{B}^{i_{k}}
\leq
\theta^{-1}\bm{B}^{i_{0}}\bm{p}\bm{q}^{-}\bm{B}^{i_{k}}.
\end{equation*}

Furthermore, we use the last inequality to write
\begin{multline*}
\theta^{-k}\bm{G}_{k}
=
\bigoplus_{0\leq i_{0}+i_{1}+\cdots+i_{k}\leq n-k-1}
\theta^{-k}(\bm{q}^{-}\bm{B}^{i_{1}}\bm{p}\cdots\bm{q}^{-}\bm{B}^{i_{k-1}}\bm{p})\bm{B}^{i_{0}}\bm{p}\bm{q}^{-}\bm{B}^{i_{k}}
\\\leq
\bigoplus_{0\leq i_{0}+i_{1}\leq n-2}
\theta^{-1}
\bm{B}^{i_{0}}\bm{p}\bm{q}^{-}\bm{B}^{i_{1}}
=
\theta^{-1}\bm{G}_{1},
\end{multline*}
from which it follows the double inequality 
\begin{equation*}
\theta^{-1}\bm{G}_{1}
\oplus
\bm{B}^{\ast}
\leq
\theta^{-1}
\bm{G}_{1}
\oplus
\bigoplus_{k=2}^{n-1}
\theta^{-k}
\bm{G}_{k}
\oplus
\bm{B}^{\ast}
\leq
\theta^{-1}\bm{G}_{1}
\oplus
\bm{B}^{\ast}.
\end{equation*}

As a result, we arrive at the matrix $\bm{G}$ in the form
\begin{equation*}
\bm{G}
=
\theta^{-1}\bm{G}_{1}
\oplus
\bm{B}^{\ast}
=
\bigoplus_{0\leq i+j\leq n-2}
\theta^{-1}
\bm{B}^{i}\bm{p}\bm{q}^{-}\bm{B}^{j}
\oplus
\bm{B}^{\ast},
\end{equation*}
which yields the formula in \eqref{E-x1}. Inequality \eqref{I-u1} is taken from \eqref{I-u} unchanged.
\end{proof}

We conclude with a discussion of the computational complexity, which is involved in the solution given by formulas \eqref{E-theta1}, \eqref{E-x1} and \eqref{I-u1}. First, we note that the evaluation of the Kleene matrix $\bm{B}^{\ast}$ takes $O(n^{3})$ operations if the Floyd–Warshall algorithm applied to the graph of the matrix $\bm{B}$ is implemented (or $O(n^{4})$ operations when using direct matrix multiplications).

Consider the calculation of the minimum $\theta$ according to \eqref{E-theta1}. We examine the first term on the right-hand side, which has the form of the sum
\begin{equation*}
\bigoplus_{0\leq i+j\leq n-2}
(\bm{h}^{-}\bm{B}^{i}\bm{p})
(\bm{q}^{-}\bm{B}^{j}\bm{g}).
\end{equation*}

To obtain the value of this sum, we calculate the vectors $\bm{B}^{i}\bm{p}$ and $\bm{B}^{j}\bm{g}$ by using matrix-vector multiplications consequently for all $i,j=0,\ldots,n-2$, which requires $O(n^{3})$ operations. With these vectors already computed, we find all summands and then the entire sum in $O(n^{2})$ time. Given the Kleene matrix $\bm{B}^{\ast}$, the second term of the form $\bm{q}^{-}\bm{B}^{\ast}\bm{p}$ requires $O(n^{2})$ operations. Considering the complexity of $O(n^{3})$ to obtain the matrix $\bm{B}^{\ast}$, we see that calculation with \eqref{E-theta1} takes no more than $O(n^{3})$ operations.

The most computationally demanding part of finding solution vectors by using \eqref{E-x1} and \eqref{I-u1} is calculating the matrix $\bm{G}$. It follows from the proof of Theorem~\ref{T-minxpqx-Bxx-gxh} that this matrix is initially defined as the Kleene matrix $\bm{G}=(\theta^{-1}\bm{p}\bm{q}^{-}\oplus\bm{B})^{\ast}$. With the minimum $\theta$ fixed, the complexity of calculating this matrix is no more than $O(n^{3})$. This complexity has the same order as that required to evaluate the minimum $\theta$, and can thus be considered as the overall complexity of the solution provided by Theorem~\ref{T-minxpqx-Bxx-gxh}.

Finally, we note that in the context of project scheduling problems, the constraint matrix $\bm{B}$ frequently appears to be sparse (have few nonzero entries), which allows one to achieve further reduction of computational complexity by using efficient algebraic techniques specifically designed for sparse matrices.

\section{Solution of project scheduling problems}
\label{S-SPSP}

We are now in a position to represent the project scheduling problems of interest in terms of tropical algebra and solve them by applying Theorem~\ref{T-minxpqx-Bxx-gxh}.

First, we consider problem \eqref{P-minmaxyimaxxi} and replace the scalar operations by the respective operations of max-plus algebra $\mathbb{R}_{\max,+}$ to rewrite the problem as
\begin{equation*}
\begin{aligned}
&
\min_{x_{i},y_{i}}
&&
\bigoplus_{1\leq i\leq n}y_{i}
\bigoplus_{1\leq l\leq n}x_{l}^{-1};
\\
&
&&
\bigoplus_{1\leq j\leq n}
b_{ij}x_{j}
\leq
x_{i},
\quad
\bigoplus_{1\leq j\leq n}
c_{ij}x_{j}
=
y_{i},
\\
&
&&
\bigoplus_{1\leq j\leq n}
d_{ij}y_{j}
\leq
x_{i},
\quad
g_{i}
\leq
x_{i}
\leq
h_{i},
\quad
y_{i}
\leq
f_{i},
\quad
i=1,\ldots,n.
\end{aligned}
\end{equation*}

To represent the problem in a vector form, we introduce the following matrices and vectors:
\begin{gather*}
\bm{B}
=
(b_{ij}),
\qquad
\bm{C}
=
(c_{ij}),
\qquad
\bm{D}
=
(d_{ij}),
\\
\bm{g}
=
(g_{i}),
\qquad
\bm{h}
=
(h_{i}),
\qquad
\bm{f}
=
(f_{i}),
\qquad
\bm{x}
=
(x_{i}),
\qquad
\bm{y}
=
(y_{i}).
\end{gather*}

With the matrix-vector notation, the problem takes the form
\begin{equation}
\begin{aligned}
&
\min_{\bm{x},\bm{y}}
&&
\bm{x}^{-}\bm{1}\bm{1}^{T}\bm{y};
\\
&
&&
\bm{B}\bm{x}
\leq
\bm{x},
\quad
\bm{C}\bm{x}
=
\bm{y},
\quad
\bm{D}\bm{y}
\leq
\bm{x},
\\
&
&&
\bm{g}
\leq
\bm{x}
\leq
\bm{h},
\quad
\bm{y}
\leq
\bm{f}.
\end{aligned}
\label{P-minxy-Bxx-Cxy-Dyx-gxh-yf}
\end{equation}

The next statement offers a complete solution to problem \eqref{P-minxy-Bxx-Cxy-Dyx-gxh-yf}.
\begin{theorem}
\label{T-minxy-Bxx-Cxy-Dyx-gxh-yf}
Let $\bm{B}$ and $\bm{D}$ be matrices and $\bm{C}$ column-regular matrix such that the matrix $\bm{R}=\bm{B}\oplus\bm{D}\bm{C}$ satisfies the condition $\mathop\mathrm{Tr}(\bm{R})\leq\mathbb{1}$.
Let $\bm{g}$ be a vector, $\bm{f}$ and $\bm{h}$ regular vectors such that the vector $\bm{s}^{-}=\bm{f}^{-}\bm{C}\oplus\bm{h}^{-}$ satisfies the condition $\bm{s}^{-}\bm{R}^{\ast}\bm{g}\leq\mathbb{1}$.
Then, the minimum value of the objective function in problem \eqref{P-minxy-Bxx-Cxy-Dyx-gxh-yf} is equal to
\begin{equation}
\theta
=
\bigoplus_{0\leq i+j\leq n-2}
\|\bm{s}^{-}\bm{R}^{i}\|
\|\bm{C}\bm{R}^{j}\bm{g}\|
\oplus
\|\bm{C}\bm{R}^{\ast}\|,
\label{E-theta2}
\end{equation}
and all regular solutions are given in the parametric form
\begin{equation}
\bm{x}
=
\bm{G}\bm{u},
\qquad
\bm{y}
=
\bm{C}\bm{G}\bm{u},
\qquad
\bm{G}
=
\bigoplus_{0 \leq i+j \leq n-2}
\theta^{-1}
\bm{R}^{i}\bm{1}
\bm{1}^{T}\bm{C}\bm{R}^{j}
\oplus
\bm{R}^{\ast},
\label{E-x2}
\end{equation}
where $\bm{u}\ne\bm{0}$ is a vector of parameters that satisfies the conditions
\begin{equation}
\bm{g}
\leq
\bm{u}
\leq
(\bm{s}^{-}\bm{G})^{-}.
\label{I-u2}
\end{equation}
\end{theorem}
\begin{proof}
We start the solution by eliminating the unknown vector $\bm{y}$ from the problem. After the substitution $\bm{y}=\bm{C}\bm{x}$, problem \eqref{P-minxy-Bxx-Cxy-Dyx-gxh-yf} reduces to the problem
\begin{equation*}
\begin{aligned}
&
\min_{\bm{x}}
&&
\bm{x}^{-}\bm{1}\bm{1}^{T}\bm{C}\bm{x};
\\
&
&&
\bm{B}\bm{x}
\leq
\bm{x},
\quad
\bm{D}\bm{C}\bm{x}
\leq
\bm{x},
\\
&
&&
\bm{g}
\leq
\bm{x}
\leq
\bm{h},
\quad
\bm{C}\bm{x}
\leq
\bm{f}.
\end{aligned}
\end{equation*}

Furthermore, we rearrange the system of constraints in the new problem. We couple the inequalities $\bm{B}\bm{x}\leq\bm{x}$ and $\bm{D}\bm{C}\bm{x}\leq\bm{x}$ to write $\bm{R}\bm{x}=(\bm{B}\oplus\bm{D}\bm{C})\bm{x}\leq\bm{x}$.

Application of Lemma~\ref{L-Axb} to the inequality $\bm{C}\bm{x}\leq\bm{f}$ leads to the inequality $\bm{x}\leq(\bm{f}^{-}\bm{C})^{-}$. To combine this inequality with the right inequality in $\bm{g}\leq\bm{x}\leq\bm{h}$, we use conjugate transposition to represent these inequalities as $\bm{x}^{-}\geq\bm{f}^{-}\bm{C}$ and $\bm{x}^{-}\geq\bm{h}^{-}$. We replace the obtained inequalities by the inequality $\bm{x}^{-}\geq\bm{f}^{-}\bm{C}\oplus\bm{h}^{-}$, and then rewrite the last inequality as $\bm{x}\leq(\bm{f}^{-}\bm{C}\oplus\bm{h}^{-})^{-}=\bm{s}$.  

As a result, the above problem turns into the problem 
\begin{equation*}
\begin{aligned}
&
\min_{\bm{x}}
&&
\bm{x}^{-}\bm{1}\bm{1}^{T}\bm{C}\bm{x};
\\
&
&&
\bm{R}\bm{x}
\leq
\bm{x},
\quad
\bm{g}
\leq
\bm{x}
\leq
\bm{s};
\end{aligned}
\end{equation*}
which takes the form of \eqref{P-minxpqx-Bxx-gxh} with $\bm{p}=\bm{1}$, $\bm{q}^{-}=\bm{1}^{T}\bm{C}$, $\bm{B}=\bm{R}$ and $\bm{h}=\bm{s}$.

We solve this problem by using Theorem~\ref{T-minxpqx-Bxx-gxh}. First, we verify that the conditions of the theorem, given by the inequalities $\mathop\mathrm{Tr}(\bm{B})\leq\mathbb{1}$, $\bm{h}^{-}\bm{B}^{\ast}\bm{g}\leq\mathbb{1}$ and $\bm{q}^{-}\bm{p}>\mathbb{0}$, are satisfied. Indeed, the first and second inequalities now become $\mathop\mathrm{Tr}(\bm{R})\leq\mathbb{1}$ and $\|\bm{s}^{-}\bm{R}^{\ast}\|=\bm{s}^{-}\bm{R}^{\ast}\bm{1}\leq\mathbb{1}$, which are valid by the current assumptions. The third inequality holds in the form $\|\bm{C}\|=\bm{1}^{T}\bm{C}\bm{1}>\mathbb{0}$ because $\bm{C}$ is a column-regular matrix. 

The application of Theorem~\ref{T-minxpqx-Bxx-gxh} requires the evaluation of the minimum $\theta$ according to the formula \eqref{E-theta1} and the description of the solution set using \eqref{E-x1} and \eqref{I-u1}. To refine the formulas, we rewrite the expressions $\bm{B}^{i}\bm{p}$ and $\bm{q}^{-}\bm{B}^{j}$, which become $\bm{R}^{i}\bm{1}$ and $\bm{1}^{T}\bm{C}\bm{R}^{j}$. Next, we represent $\bm{q}^{-}\bm{B}^{\ast}\bm{p}$, $\bm{h}^{-}\bm{B}^{i}\bm{p}$ and $\bm{q}^{-}\bm{B}^{j}\bm{g}$ respectively as $\bm{1}^{T}\bm{C}\bm{R}^{\ast}\bm{1}=\|\bm{C}\bm{R}^{\ast}\|$, $\bm{s}^{-}\bm{R}^{i}\bm{1}=\|\bm{s}^{-}\bm{R}^{i}\|$ and $\bm{1}^{T}\bm{C}\bm{R}^{j}\bm{g}=\|\bm{C}\bm{R}^{j}\bm{g}\|$.

Substitution of the above expressions yields \eqref{E-theta2}, \eqref{E-x2} and \eqref{I-u2}.
\end{proof}
 
Finally, we consider problem \eqref{P-minmaxximaxxi} and note that it has the same system of constraints as the previous problem. The objective function is different, given in the max-plus algebra setting by
\begin{equation*}
\bigoplus_{1\leq i\leq n}x_{i}
\bigoplus_{1\leq l\leq n}x_{l}^{-1}
=
\bm{x}^{-}\bm{1}\bm{1}^{T}\bm{x}.
\end{equation*}

In vector form the problem is written as 
\begin{equation}
\begin{aligned}
&
\min_{\bm{x},\bm{y}}
&&
\bm{x}^{-}\bm{1}\bm{1}^{T}\bm{x};
\\
&
&&
\bm{B}\bm{x}
\leq
\bm{x},
\quad
\bm{C}\bm{x}
=
\bm{y},
\\
&
&&
\bm{D}\bm{y}
\leq
\bm{x},
\quad
\bm{g}
\leq
\bm{x}
\leq
\bm{h},
\quad
\bm{y}
\leq
\bm{f}.
\end{aligned}
\label{P-minxx-Bxx-Cxy-Dyx-gxh-yf}
\end{equation}

We solve the problem in the same way as \eqref{P-minxy-Bxx-Cxy-Dyx-gxh-yf} to obtain the following result.
\begin{theorem}
\label{T-minxx-Bxx-Cxy-Dyx-gxh-yf}
Under the same conditions as in Theorem~\ref{T-minxy-Bxx-Cxy-Dyx-gxh-yf}, the minimum value of the objective function in problem \eqref{P-minxx-Bxx-Cxy-Dyx-gxh-yf} is equal to
\begin{equation}
\theta
=
\bigoplus_{0\leq i+j\leq n-2}
\|\bm{s}^{-} \bm{R}^{i} \|
\|\bm{R}^{j}\bm{g}\|
\oplus
\|\bm{R}^{\ast}\|,
\label{E-theta3}
\end{equation}
and all regular solutions are given in the parametric form
\begin{equation}
\bm{x}
=
\bm{G}\bm{u},
\qquad
\bm{y}
=
\bm{C}\bm{G}\bm{u},
\qquad
\bm{G}
=
\bigoplus_{0 \leq i+j \leq n-2}
\theta^{-1}
\bm{R}^{i}\bm{1}
\bm{1}^{T}\bm{R}^{j}
\oplus
\bm{R}^{\ast},
\label{E-x3}
\end{equation}
where $\bm{u}\ne\bm{0}$ is a vector of parameters that satisfies the conditions
\begin{equation}
\bm{g}
\leq
\bm{u}
\leq
(\bm{s}^{-}\bm{G})^{-}.
\label{I-u3}
\end{equation}
\end{theorem}

\section{Application to vaccination scheduling}
\label{S-AVS}

The purpose of this section is to illustrate the obtained results in a real-world context of the healthcare industry. We consider problems of temporal scheduling of vaccination activities in a medical (vaccination) centre that plans to administer vaccines to its patients (see \cite{Department2006Immunisation,Kroger2022General} for further detail on vaccine administration and scheduling).

The vaccination plan involves immunizing each patient with one or more prescribed vaccines. The immunization session of a patient with a vaccine is assumed to start with administering the vaccine dose using a dedicated (injectable, oral or intranasal) route. The patient remains under observation for possible adverse reactions at the patient treatment area for a time determined by the vaccine, after which the session is considered completed. 

Suppose there are $n$ vaccine doses to be administered in the centre. For each dose $i=1,\ldots,n$, we denote by $x_{i}$ the start time and by $y_{i}$ the finish time of the immunization session to be scheduled. Furthermore, we discuss how constraints \eqref{I-maxbijxj-xi}--\eqref{I-gixihi-yifi} in the general scheduling problems under study can be interpreted in the framework of vaccination activities.

First, we note that immunization of many patients at one medical site must follow severe requirements to minimize risk for disease exposure and spread. The required measures include cleaning and disinfecting equipment, supplies and workspaces, changing and disposing gloves, needles, syringes, as well as carrying out other procedures between each patient treatment. As a result, the start times of sessions must satisfy the inequalities $b_{ij}+x_{j}\leq x_{i}$ where $b_{ij}$ is a positive parameter that represents the minimal break time required to ensure a safe transition from one vaccination to another.

At the same time, we take into account the time frame typically specified by the vaccine manufacturer or defined by technical requirements for the vaccine to remain safe and capable once the vaccine is drawn up and/or reconstituted. This leads to the above inequalities with a negative parameter $b_{ij}$ (such that the maximum time interval between $x_{i}$ and $x_{j}$ is given by $-b_{ij}$). By combining all these inequalities, we obtain the constraints in the form of \eqref{I-maxbijxj-xi}. 

Furthermore, the duration of each vaccination session measured as the time between the start and finish of the session cannot be less than the combined time taken for administration of the dose and observation of adverse reactions. Given a parameter $c_{ij}$ equal to the minimal time slot to be allocated for session $i$ if $j=i$, or set to $-\infty$ otherwise, we have inequalities $c_{ij}+x_{j}\leq y_{i}$, which further combine into constraint \eqref{E-maxcijxj-xi}.

For certain pairs of vaccines or patients, it may be necessary to exclude the overlap in time of vaccination sessions by shifting forward the start of one session relative to the finish of another to provide time, for example, for advanced disinfection procedures or equipment operations. This condition requires that the inequalities $d_{ij}+y_{j}\leq x_{i}$ are fulfilled, where $d_{ij}$ is minimum shift length, which leads to constraint \eqref{I-maxdijxj-xi}.

Finally, we consider box constraints that may be imposed on the start time of the immunization session in the form of a time window with a lower limit $g_{i}$ (release time) and upper limit $h_{i}$ (release deadline) as a result of technological, organizational, managerial or other regulations and restrictions. Together with a natural deadline $f_{i}$ given for the completion of the session, we arrive at \eqref{I-gixihi-yifi}.

A common objective of the optimal scheduling of vaccination sessions is to minimize the makespan of the vaccination plan, defined as the time between the earliest start time and the latest finish time over all sessions and written as \eqref{E-maxyiminxi}. Combining with the above described constraints yields a project scheduling problem in the form of \eqref{P-minmaxyimaxxi} or, in terms of tropical algebra, a problem at \eqref{P-minxy-Bxx-Cxy-Dyx-gxh-yf}. A complete solution to the later problem is provided by Theorem~\ref{T-minxy-Bxx-Cxy-Dyx-gxh-yf}.

We now present a numerical example that demonstrates computational technique involved in the solution offered by Theorem~\ref{T-minxy-Bxx-Cxy-Dyx-gxh-yf}.

\begin{example}
Consider a problem of optimal vaccination scheduling according to the minimal makespan objective, which is formulated in the framework of max-plus algebra $\mathbb{R}_{\max,+}$ as \eqref{P-minxy-Bxx-Cxy-Dyx-gxh-yf}. Suppose that the temporal constraints are given by the following matrices and vectors (where the symbol $\mathbb{0}=-\infty$ is used for the sake of compactness):
\begin{gather*}
\bm{B}
=
\left(
\begin{array}{ccrcr}
0 & \mathbb{0} & \mathbb{0} & 0 & \mathbb{0}
\\
1 & 0 & \mathbb{0} & \mathbb{0} & \mathbb{0}
\\
\mathbb{0} & \mathbb{0} & 0 & 1 & -1
\\
0 & \mathbb{0} & \mathbb{0} & 0 & \mathbb{0}
\\
\mathbb{0} & \mathbb{0} & -1 & \mathbb{0} & 0
\end{array}
\right),
\qquad
\bm{C}
=
\left(
\begin{array}{ccccc}
4 & \mathbb{0} & \mathbb{0} & \mathbb{0} & \mathbb{0}
\\
\mathbb{0} & 4 & \mathbb{0} & \mathbb{0} & \mathbb{0}
\\
\mathbb{0} & \mathbb{0} & 5 & \mathbb{0} & \mathbb{0}
\\
\mathbb{0} & \mathbb{0} & \mathbb{0} & 5 & \mathbb{0}
\\
\mathbb{0} & \mathbb{0} & \mathbb{0} & \mathbb{0} & 3
\end{array}
\right),
\qquad
\bm{D}
=
\left(
\begin{array}{ccccc}
\mathbb{0} & \mathbb{0} & \mathbb{0} & \mathbb{0} & \mathbb{0}
\\
\mathbb{0} & \mathbb{0} & \mathbb{0} & \mathbb{0} & \mathbb{0}
\\
0 & \mathbb{0} & \mathbb{0} & \mathbb{0} & \mathbb{0}
\\
\mathbb{0} & \mathbb{0} & \mathbb{0} & \mathbb{0} & \mathbb{0}
\\
\mathbb{0} & 0 & \mathbb{0} & 0 & \mathbb{0}
\end{array}
\right),
\\
\bm{g}
=
\left(
\begin{array}{c}
0
\\
0
\\
0
\\
0
\\
0
\end{array}
\right),
\qquad
\bm{h}
=
\left(
\begin{array}{c}
4
\\
5
\\
8
\\
9
\\
5
\end{array}
\right),
\qquad
\bm{f}
=
\left(
\begin{array}{c}
12
\\
12
\\
12
\\
12
\\
12
\end{array}
\right).
\end{gather*}

The constraints defined by the matrices have the following interpretation. Consider the matrix $\bm{B}$ and examine those entries which are not equal to $\mathbb{0}$. First note that the diagonal entries trivially set to arithmetic zero since $x_{i}\leq x_{i}$ for all $i$. The nonzero off-diagonal entries of the matrix $\bm{B}$ yield the inequalities
\begin{equation*}
x_{4}\leq x_{1},
\quad
1x_{1}\leq x_{2},
\quad
1x_{4}\leq x_{3},
\quad
(-1)x_{5}\leq x_{3},
\quad
x_{1}\leq x_{4},
\quad
(-1)x_{3}\leq x_{5}.
\end{equation*}

These conditions establish the relationships between the start times of vaccination sessions, which read in conventional notation as 
\begin{equation*}
x_{1}=x_{4},
\qquad
x_{1}+1\leq x_{2},
\qquad
\lvert x_{3}-x_{5}\rvert\leq1,
\qquad
x_{4}+1\leq x_{3}.
\end{equation*}

The diagonal entries of the matrix $\bm{C}$ specifies the minimum duration of vaccination sessions consisting of vaccine administration and patient observation. 

The entries in the matrix $\bm{D}$, which differ from $\mathbb{0}$, set the following relationships between the start and finish times of sessions:
\begin{equation*}  
y_{1}\leq x_{3},
\qquad
y_{2}\leq x_{5},
\qquad
y_{4}\leq x_{5}.
\end{equation*}  

To solve the problem by applying Theorem~\ref{T-minxy-Bxx-Cxy-Dyx-gxh-yf}, we need to verify the conditions of the theorem. For this purpose, we evaluate the matrices 
\begin{equation*}
\bm{D}\bm{C}
=
\left(
\begin{array}{ccccc}
\mathbb{0} & \mathbb{0} & \mathbb{0} & \mathbb{0} & \mathbb{0}
\\
\mathbb{0} & \mathbb{0} & \mathbb{0} & \mathbb{0} & \mathbb{0}
\\
4 & \mathbb{0} & \mathbb{0} & \mathbb{0} & \mathbb{0}
\\
\mathbb{0} & \mathbb{0} & \mathbb{0} & \mathbb{0} & \mathbb{0}
\\
\mathbb{0} & 4 & \mathbb{0} & 5 & \mathbb{0}
\end{array}
\right),
\qquad
\bm{R}
=
\bm{B}\oplus\bm{D}\bm{C}
=
\left(
\begin{array}{ccrcr}
0 & \mathbb{0} & \mathbb{0} & 0 & \mathbb{0}
\\
1 & 0 & \mathbb{0} & \mathbb{0} & \mathbb{0}
\\
4 & \mathbb{0} & 0 & 1 & -1
\\
0 & \mathbb{0} & \mathbb{0} & 0 & \mathbb{0}
\\
\mathbb{0} & 4 & -1 & 5 & 0
\end{array}
\right).
\end{equation*}

After the calculation of powers of the matrix $\bm{R}$, we obtain
\begin{equation*}
\bm{R}^{2}
=
\bm{R}^{3}
=
\bm{R}^{4}
=
\bm{R}^{5}
=
\left(
\begin{array}{ccrcr}
0 & \mathbb{0} & \mathbb{0} & 0 & \mathbb{0}
\\
1 & 0 & \mathbb{0} & 1 & \mathbb{0}
\\
4 & 3 & 0 & 4 & -1
\\
0 & \mathbb{0} & \mathbb{0} & 0 & \mathbb{0}
\\
5 & 4 & -1 & 5 & 0
\end{array}
\right).
\end{equation*}

Since $\mathop\mathrm{tr}\bm{R}^{k}=0$ for all $k=1,\ldots,5$, we have $\mathop\mathrm{Tr}(\bm{R})
=\mathop\mathrm{tr}\bm{R}\oplus\cdots\oplus\mathop\mathrm{tr}\bm{R}^{5}=0=\mathbb{1}$, and thus we conclude that the condition $\mathop\mathrm{Tr}(\bm{R})\leq\mathbb{1}$ is satisfied.

Furthermore, we form the Kleene matrix
\begin{equation*}
\bm{R}^{\ast}
=
\bm{I}\oplus\bm{R}\oplus\cdots\oplus\bm{R}^{4}
=
\left(
\begin{array}{ccrcr}
0 & \mathbb{0} & \mathbb{0} & 0 & \mathbb{0}
\\
1 & 0 & \mathbb{0} & 1 & \mathbb{0}
\\
4 & 3 & 0 & 4 & -1
\\
0 & \mathbb{0} & \mathbb{0} & 0 & \mathbb{0}
\\
5 & 4 & -1 & 5 & 0
\end{array}
\right).
\end{equation*}

We successfully calculate the row vectors
\begin{gather*}
\bm{f}^{-}\bm{C}
=
\left(
\begin{array}{rrrrr}
-8 & -8 & -7 & -10 & -9
\end{array}
\right),
\qquad
\bm{s}^{-}
=
\bm{f}^{-}\bm{C}
\oplus
\bm{h}^{-}
=
\left(
\begin{array}{rrrrr}
-4 & -5 & -7 & -9 & -5
\end{array}
\right),
\\
\bm{s}^{-}\bm{R}^{\ast}
=
\left(
\begin{array}{crrcr}
0 & -1 & -6 & 0 & -5
\end{array}
\right),
\end{gather*}
and the scalar $\bm{s}^{-}\bm{R}^{\ast}\bm{g}=0=\mathbb{1}$, which shows that the condition $\bm{s}^{-}\bm{R}^{\ast}\bm{g}\leq\mathbb{1}$ holds.

We now evaluate the minimum $\theta$ of the objective function, defined by \eqref{E-theta2} as
\begin{multline*}
\theta
=
\|\bm{s}^{-}\|(\|\bm{C}\bm{g}\|\oplus\|\bm{C}\bm{R}\bm{g}\|\oplus\|\bm{C}\bm{R}^{2}\bm{g}\|\oplus\|\bm{C}\bm{R}^{3}\bm{g}\|)
\\
\oplus
\|\bm{s}^{-}\bm{R}\|(\|\bm{C}\bm{g}\|\oplus\|\bm{C}\bm{R}\bm{g}\|\oplus\|\bm{C}\bm{R}^{2}\bm{g}\|)
\oplus
\|\bm{s}^{-}\bm{R}^{2}\|(\|\bm{C}\bm{g}\|\oplus\|\bm{C}\bm{R}\bm{g}\|)
\\
\oplus
\|\bm{s}^{-}\bm{R}^{3}\|\|\bm{C}\bm{g}\|
\oplus
\|\bm{C}\bm{R}^{\ast}\|.
\end{multline*}

First, we obtain the row vectors
\begin{equation*}
\bm{s}^{-}\bm{R}
=
\left(
\begin{array}{rrrcr}
-3 & -1 & -6 & 0 & -5
\end{array}
\right),
\qquad
\bm{s}^{-}\bm{R}^{2}
=
\bm{s}^{-}\bm{R}^{3}
=
\left(
\begin{array}{crrcr}
0 & -1 & -6 & 0 & -5
\end{array}
\right).
\end{equation*}

Next, we find the matrices
\begin{equation*}
\bm{C}\bm{R}
=
\left(
\begin{array}{ccccc}
4 & \mathbb{0} & \mathbb{0} & 4 & \mathbb{0}
\\
5 & 4 & \mathbb{0} & \mathbb{0} & \mathbb{0}
\\
9 & \mathbb{0} & 5 & 6 & 4
\\
5 & \mathbb{0} & \mathbb{0} & 5 & \mathbb{0}
\\
\mathbb{0} & 7 & 2 & 8 & 3
\end{array}
\right),
\qquad
\bm{C}\bm{R}^{2}
=
\bm{C}\bm{R}^{3}
=
\left(
\begin{array}{ccccc}
4 & \mathbb{0} & \mathbb{0} & 4 & \mathbb{0}
\\
5 & 4 & \mathbb{0} & 5 & \mathbb{0}
\\
9 & 8 & 5 & 9 & 4
\\
5 & \mathbb{0} & \mathbb{0} & 5 & \mathbb{0}
\\
8 & 7 & 2 & 8 & 3
\end{array}
\right),
\end{equation*}
and then calculate the column vectors 
\begin{equation*}
\bm{C}\bm{g}
=
\left(
\begin{array}{c}
4
\\
4
\\
5
\\
5
\\
3
\end{array}
\right),
\qquad
\bm{C}\bm{R}\bm{g}
=
\left(
\begin{array}{c}
4
\\
5
\\
9
\\
5
\\
8
\end{array}
\right),
\qquad
\bm{C}\bm{R}^{2}\bm{g}
=
\bm{C}\bm{R}^{3}\bm{g}
=
\left(
\begin{array}{c}
4
\\
5
\\
9
\\
5
\\
8
\end{array}
\right).
\end{equation*}

We evaluate the norms to obtain the results
\begin{gather*}
\|\bm{s}^{-}\|
=
-4,
\qquad
\|\bm{s}^{-}\bm{R}\|
=
\|\bm{s}^{-}\bm{R}^{2}\|
=
\|\bm{s}^{-}\bm{R}^{3}\|
=
0,
\\
\|\bm{C}\bm{g}\|
=
5.
\qquad
\|\bm{C}\bm{R}\bm{g}\|
=
\|\bm{C}\bm{R}^{2}\bm{g}\|
=
\|\bm{C}\bm{R}^{3}\bm{g}\|
=
\|\bm{C}\bm{R}^{\ast}\|
=
9.
\end{gather*}

Finally, substitution of these results yields the minimum
\begin{equation*}
\theta
=
9,
\end{equation*}
which specifies the minimum makespan of the optimal vaccination schedule.

We complete the solution with the determination of the start and finish times of vaccination sessions in the optimal schedule. We begin with the evaluation of the generating matrix $\bm{G}$ according to \eqref{E-x2}, which takes the form
\begin{multline*}
\bm{G}
=
\theta^{-1}
(
\bm{1}
\bm{1}^{T}\bm{C}
\oplus
\bm{1}
\bm{1}^{T}\bm{C}\bm{R}
\oplus
\bm{1}
\bm{1}^{T}\bm{C}\bm{R}^{2}
\oplus
\bm{1}
\bm{1}^{T}\bm{C}\bm{R}^{3}
\\
\oplus
\bm{R}\bm{1}
\bm{1}^{T}\bm{C}
\oplus
\bm{R}\bm{1}
\bm{1}^{T}\bm{C}\bm{R}
\oplus
\bm{R}\bm{1}
\bm{1}^{T}\bm{C}\bm{R}^{2}
\\
\oplus
\bm{R}^{2}\bm{1}
\bm{1}^{T}\bm{C}
\oplus
\bm{R}^{2}\bm{1}
\bm{1}^{T}\bm{C}\bm{R}
\oplus
\bm{R}^{3}\bm{1}
\bm{1}^{T}\bm{C}
)
\oplus
\bm{R}^{\ast}.
\end{multline*}

Furthermore, we consider the column vectors
\begin{equation*}
\bm{1}
=
\left(
\begin{array}{c}
0
\\
0
\\
0
\\
0
\\
0
\end{array}
\right),
\qquad
\bm{R}\bm{1}
=
\bm{R}^{2}\bm{1}
=
\bm{R}^{3}\bm{1}
=
\left(
\begin{array}{c}
0
\\
1
\\
4
\\
0
\\
5
\end{array}
\right),
\qquad
\end{equation*}
and the row vectors
\begin{gather*}
\bm{1}^{T}\bm{C}
=
\left(
\begin{array}{ccccc}
4 & 4 & 5 & 5 & 3
\end{array}
\right),
\qquad
\bm{1}^{T}\bm{C}\bm{R}
=
\left(
\begin{array}{ccccc}
9 & 7 & 5 & 8 & 4
\end{array}
\right),
\\
\bm{1}^{T}\bm{C}\bm{R}^{2}
=
\bm{1}^{T}\bm{C}\bm{R}^{3}
=
\left(
\begin{array}{ccccc}
9 & 8 & 5 & 9 & 4
\end{array}
\right).
\end{gather*}

Multiplying these vectors gives the matrices
\begin{gather*}
\bm{1}\bm{1}^{T}\bm{C}
=
\left(
\begin{array}{ccccc}
4 & 4 & 5 & 5 & 3
\\
4 & 4 & 5 & 5 & 3
\\
4 & 4 & 5 & 5 & 3
\\
4 & 4 & 5 & 5 & 3
\\
4 & 4 & 5 & 5 & 3
\end{array}
\right),
\qquad
\bm{1}\bm{1}^{T}\bm{C}\bm{R}
=
\left(
\begin{array}{ccccc}
9 & 7 & 5 & 8 & 4
\\
9 & 7 & 5 & 8 & 4
\\
9 & 7 & 5 & 8 & 4
\\
9 & 7 & 5 & 8 & 4
\\
9 & 7 & 5 & 8 & 4
\end{array}
\right),
\\
\bm{1}\bm{1}^{T}\bm{C}\bm{R}^{2}
=
\bm{1}\bm{1}^{T}\bm{C}\bm{R}^{3}
=
\left(
\begin{array}{ccccc}
9 & 8 & 5 & 9 & 4
\\
9 & 8 & 5 & 9 & 4
\\
9 & 8 & 5 & 9 & 4
\\
9 & 8 & 5 & 9 & 4
\\
9 & 8 & 5 & 9 & 4
\end{array}
\right),
\qquad
\bm{R}\bm{1}\bm{1}^{T}\bm{C}\bm{R}^{2}
=
\left(
\begin{array}{ccccc}
9 & 8 & 5 & 9 & 4
\\
10 & 9 & 6 & 10 & 5
\\
13 & 12 & 9 & 13 & 8
\\
9 & 8 & 5 & 9 & 4
\\
14 & 13 & 10 & 14 & 9
\end{array}
\right).
\\
\bm{R}\bm{1}\bm{1}^{T}\bm{C}
=
\bm{R}^{2}\bm{1}\bm{1}^{T}\bm{C}
=
\bm{R}^{3}\bm{1}\bm{1}^{T}\bm{C}
=
\left(
\begin{array}{ccccc}
4 & 4 & 5 & 5 & 3
\\
5 & 5 & 6 & 6 & 4
\\
8 & 8 & 9 & 9 & 7
\\
4 & 4 & 5 & 5 & 3
\\
9 & 9 & 10 & 10 & 8
\end{array}
\right),
\\
\bm{R}\bm{1}\bm{1}^{T}\bm{C}\bm{R}
=
\bm{R}^{2}\bm{1}\bm{1}^{T}\bm{C}\bm{R}
=
\left(
\begin{array}{ccccc}
9 & 7 & 5 & 8 & 4
\\
10 & 8 & 6 & 9 & 5
\\
13 & 11 & 9 & 12 & 8
\\
9 & 7 & 5 & 8 & 4
\\
14 & 12 & 10 & 13 & 9
\end{array}
\right).
\end{gather*}

After substitution of these matrices together with the matrix $\bm{R}^{\ast}$ and $\theta=9$, we arrive at the generating matrix
\begin{equation*}
\bm{G}
=
\left(
\begin{array}{crrcr}
0 & -1 & -4 & 0 & -5
\\
1 & 0 & -3 & 1 & -4
\\
4 & 3 & 0 & 4 & -1
\\
0 & -1 & -4 & 0 & -5
\\
5 & 4 & 1 & 5 & 0
\end{array}
\right).
\end{equation*}

According to \eqref{I-u2}, the vector of parameters $\bm{u}$ satisfies the double inequality $\bm{u}_{1}\leq\bm{u}\leq\bm{u}_{2}$ with the bounds given by $\bm{u}_{1}=\bm{g}$ and $\bm{u}_{2}=(\bm{s}^{-}\bm{G})^{-}$. Calculation of these bounds and corresponding solution vectors $\bm{x}_{1}=\bm{G}\bm{u}_{1}$ and $\bm{x}_{2}=\bm{G}\bm{u}_{2}$ yields
\begin{equation*}
\bm{u}_{1}
=
\left(
\begin{array}{c}
0
\\
0
\\
0
\\
0
\\
0
\end{array}
\right),
\qquad
\bm{u}_{2}
=
\left(
\begin{array}{c}
0
\\
1
\\
4
\\
0
\\
5
\end{array}
\right),
\qquad
\bm{x}_{1}
=
\bm{x}_{2}
=
\left(
\begin{array}{c}
0
\\
1
\\
4
\\
0
\\
5
\end{array}
\right).
\end{equation*}

Observing that both solutions coincide, we conclude that the problem has a unique solution $\bm{x}$, which together with the vector $\bm{y}=\bm{C}\bm{x}$ takes the form
\begin{equation*}
\bm{x}
=
\left(
\begin{array}{c}
0
\\
1
\\
4
\\
0
\\
5
\end{array}
\right),
\qquad
\bm{y}
=
\left(
\begin{array}{c}
4
\\
5
\\
9
\\
5
\\
8
\end{array}
\right).
\end{equation*}
 
These vectors $\bm{x}$ and $\bm{y}$ determine the start and finish time of vaccination sessions
\begin{gather*}
x_{1}=0,
\qquad
x_{2}=1,
\qquad
x_{3}=4,
\qquad
x_{4}=0,
\qquad
x_{5}=5,
\\
y_{1}=4,
\qquad
y_{2}=5,
\qquad
y_{3}=9,
\qquad
y_{4}=5,
\qquad
y_{5}=8,
\end{gather*}
which provide an optimal schedule that minimizes the vaccination makespan.

A graphical representation of the optimal schedule in the form of a Gantt chart is given in Figure~\ref{F-GCOSE1}.
\begin{figure}[ht]
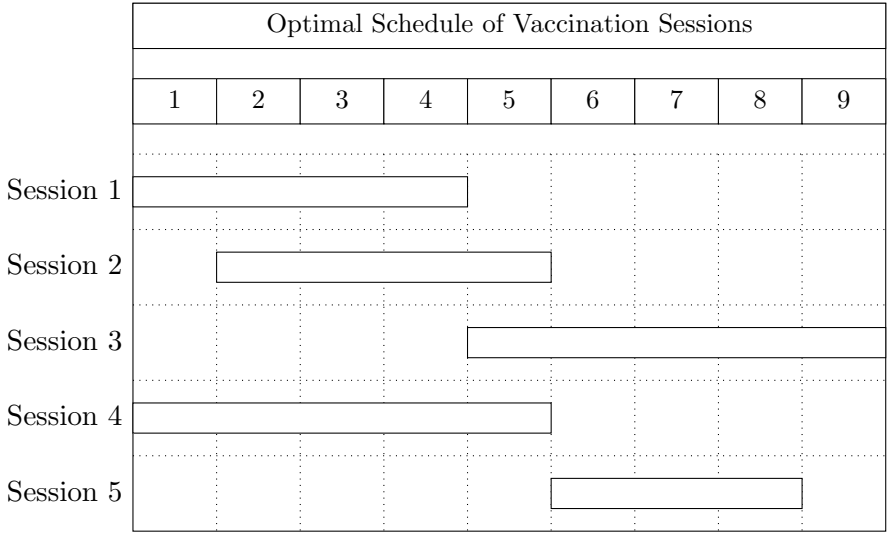

\begin{ganttchart}[
hgrid,
vgrid,
x unit=1.1cm,
y unit title=1cm,
y unit chart=1cm
]{1}{9}
\gantttitle{Optimal Schedule of Vaccination Sessions}{9} \\
\gantttitlelist{1,...,9}{1} \\
\ganttbar{Session 1}{1}{4} \\
\ganttbar{Session 2}{2}{5} \\
\ganttbar{Session 3}{5}{9} \\
\ganttbar{Session 4}{1}{5} \\
\ganttbar{Session 5}{6}{8}
\end{ganttchart}
\caption{The Gantt chart of the optimal schedule of $n=5$ vaccination sessions for Example 1}
\label{F-GCOSE1}
\end{figure}

As it is easy to see, the obtained solutions satisfy all constraints in the problem.
\end{example}

\begin{example}
We now assume that under the same constraints as in the previous example, the problem is to minimize the maximum absolute deviation of start times. This problem takes the form of problem \eqref{P-minxx-Bxx-Cxy-Dyx-gxh-yf}, which is solved by Theorem~\ref{T-minxx-Bxx-Cxy-Dyx-gxh-yf}.

We omit the details of the solution, which follows the same steps as that of \eqref{P-minxy-Bxx-Cxy-Dyx-gxh-yf} and go straight to the results. Evaluation of the minimum $\theta$ according to \eqref{E-theta3}, which shows the minimal deviation of the start times of vaccination sessions yields
\begin{equation*}
\theta=5.
\end{equation*}

The generating matrix to calculate solutions at \eqref{E-x3} takes the form
\begin{equation*}
\bm{G}
=
\left(
\begin{array}{crrcr}
0 & -1 & -5 & 0 & -5
\\
1 & 0 & -4 & 1 & -4
\\
4 & 3 & 0 & 4 & -1
\\
0 & -1 & -5 & 0 & -5
\\
5 & 4 & 0 & 5 & 0
\end{array}
\right).
\end{equation*}

We evaluate the lower and upper bounds $\bm{u}_{1}=\bm{g}$ and $\bm{u}_{2}=(\bm{s}^{-}\bm{G})^{-}$ in \eqref{I-u3} for the parameter vector $\bm{u}$ to find
\begin{equation*}
\bm{u}_{1}
=
\left(
\begin{array}{c}
0
\\
0
\\
0
\\
0
\\
0
\end{array}
\right),
\qquad
\bm{u}_{2}
=
\left(
\begin{array}{c}
0
\\
1
\\
5
\\
0
\\
5
\end{array}
\right).
\end{equation*}

The solution vectors of start time $\bm{x}_{1}=\bm{G}\bm{u}_{1}$ and $\bm{x}_{2}=\bm{G}\bm{u}_{2}$, and the corresponding vectors of finish time $\bm{y}_{1}=\bm{C}\bm{x}_{1}$ and $\bm{y}_{2}=\bm{C}\bm{x}_{2}$ are written as
\begin{equation*}
\bm{x}_{1}
=
\left(
\begin{array}{c}
0
\\
1
\\
4
\\
0
\\
5
\end{array}
\right),
\qquad
\bm{x}_{2}
=
\left(
\begin{array}{c}
0
\\
1
\\
5
\\
0
\\
5
\end{array}
\right),
\qquad
\bm{y}_{1}
=
\left(
\begin{array}{c}
4
\\
5
\\
9
\\
5
\\
8
\end{array}
\right),
\qquad
\bm{y}_{2}
=
\left(
\begin{array}{c}
4
\\
5
\\
10
\\
5
\\
8
\end{array}
\right).
\end{equation*}

Since the vectors in both pairs differ in only one component, the optimal start and finish times of sessions, which minimize the maximum absolute deviation (variation) of the start times, are given by the conditions
\begin{gather*}
x_{1}=0,
\qquad
x_{2}=1,
\qquad
4\leq x_{3}\leq5,
\qquad
x_{4}=0,
\qquad
x_{5}=5,
\\
y_{1}=4,
\qquad
y_{2}=5,
\qquad
y_{3}=x_{3}+5,
\qquad
y_{4}=5,
\qquad
y_{5}=8.
\end{gather*}

Note that if we take $x_{3}=4$, the obtained optimal schedule coincides with that found in the previous example. A common solution of these examples seems to be quite consistent with that as the deviation of start time decreases, the makespan of the schedule has a general tendency to decrease as well.

Figure~\ref{F-GCOSE2} shows another solution where the optimal schedule has the start and finish times of the third session defined as $x_{3}=5$ and $y_{3}=10$.
\begin{figure}[ht]
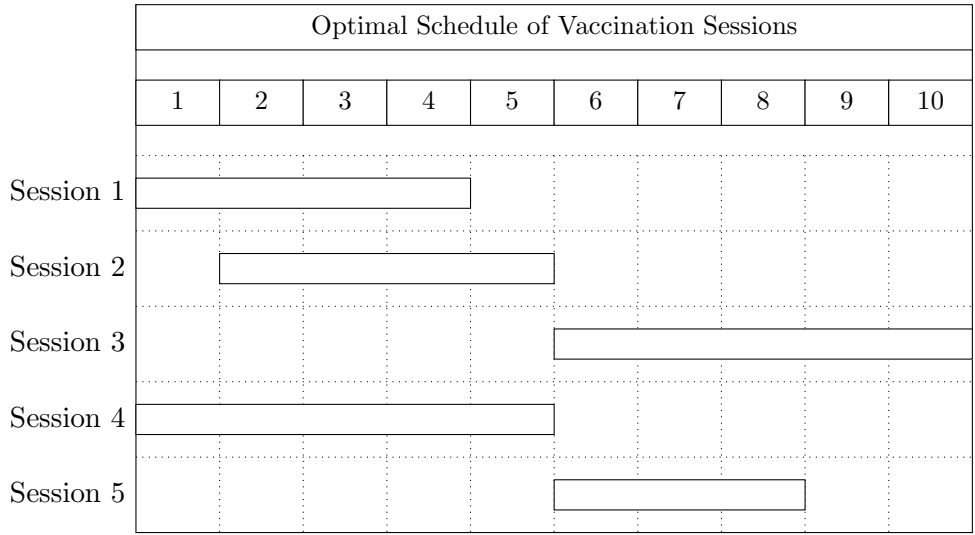

\begin{ganttchart}[
hgrid,
vgrid,
x unit=1.1cm,
y unit title=1cm,
y unit chart=1cm
]{1}{10}
\gantttitle{Optimal Schedule of Vaccination Sessions}{10} \\
\gantttitlelist{1,...,10}{1} \\
\ganttbar{Session 1}{1}{4} \\
\ganttbar{Session 2}{2}{5} \\
\ganttbar{Session 3}{6}{10} \\
\ganttbar{Session 4}{1}{5} \\
\ganttbar{Session 5}{6}{8}
\end{ganttchart}
\caption{The Gantt chart of the optimal schedule of $n=5$ vaccination sessions for Example 2}
\label{F-GCOSE2}
\end{figure}
\end{example}

\section{Conclusions}
\label{S-C}

This study has developed a new algebraic approach to derive direct analytical solutions for temporal project scheduling problems. We have considered a project that consists of a set of activities that are undertaken in parallel under time-based constraints (``start--start'', ``start--finish'' and ``finish--start'' precedence relations, release times, release deadlines and completion deadlines) to achieve time-based objectives (minimizing the project makespan or the maximum absolute deviation of start times of activities). Such problems occur as auxiliary tools to handle general scheduling problems that may involve non-temporal constraints (in terms of manpower, machines, tools, equipment, space, money), and are of independent interest in project management.

A common approach that can be used to solve the above problems is to represent them as linear programs. The problems are then numerically solved using computational methods and techniques of linear programming (simplex method, Karmarkar algorithm), which use iterative computational procedures to obtain one of the solutions if they exist. However, this algorithmic approach does not allow to derive a direct analytical solution of the scheduling problems that is capable to describe all optimal schedules in a compact vector form.
 
To obtain an analytical solution in this study, we have used an approach based on moving from the conventional algebra to the tropical (idempotent) algebra focused on the theory and applications of algebraic systems with idempotent addition. We have started with a short self-contained introduction to tropical algebra so as to make further discussion easier to follow. We have considered a new general constrained optimization problem formulated in terms of tropical algebra, and derived a complete analytical solution to the problem. Finally, the obtained result has been exploited to solve the project scheduling problems of interest, represented in the tropical algebra setting.

As an outcome, new direct solutions to the scheduling problems under study have been found analytically as expressions that concisely describe all optimal schedules in parametric vector form. In contrast to the known algorithmic solutions provided by linear programming, the analytical form of the results is more insightful, handy and flexible. It provides more opportunities for formal analysis of the solutions, can tell more about the fundamental structure of the solution set and better explain the impact of the input parameters on the solutions. Moreover, the analytical representation offers a potential to refine the solution by taking into account additional constraints and objectives, and extend the approach to multicriteria scheduling problems. 

Furthermore, the analytical solution obtained is more straightforward and much simpler to implement. The solution procedure consists in a finite number of matrix-vector operations and is ready for efficient software implementation to run on both serial and parallel computers. This solution is more accurate since it avoids the accumulation of rounding errors that are usually involved in numerical algorithms. At the same time, the computational complexity of the new solution, estimated as $O(n^{3})$ arithmetic operations, is less than the complexity of the Karmarkar algorithm, which is known to be $O(n^{3.5}L)$, where $L$ is the number of bits in the input \cite{Karmarkar1984New}.

Real-world problems of optimal scheduling of vaccination activities in a medical centre were described to emphasize the practical relevance and value of the research. In particular, we have shown what meaning the formal constraints and objectives of the scheduling model used can have in a real situation. Finally, simple illustrative examples were presented in some detail to demonstrate computational technique involved in the solution obtained. 

As it follows from the above remarks and discussion, the proposed algebraic solutions appear a useful complement and supplement to existing algorithmic solutions and become more desirable if an analytical result is of interest.

\bibliographystyle{abbrvurl}
\bibliography{Algebraic_solution_of_project_scheduling_problems_with_temporal_constraints}

\begin{thebibliography}{10}

\bibitem{Ballestin2015Handbook}
F.~Ballest{\'i}n and R.~Blanco.
\newblock Theoretical and practical fundamentals.
\newblock In C.~Schwindt and J.~Zimmermann, editors, {\em Handbook on Project
  Management and Scheduling Vol.1}, International Handbooks on Information
  Systems, pages 411--427. Springer, Cham, 2015.
\newblock \href {https://doi.org/10.1007/978-3-319-05443-8_19}
  {\path{doi:10.1007/978-3-319-05443-8_19}}.

\bibitem{Bouquard2006Application}
J.-L. Bouquard, C.~Lent\'{e}, and J.-C. Billaut.
\newblock Application of an optimization problem in max-plus algebra to
  scheduling problems.
\newblock {\em Discrete Appl. Math.}, 154(15):2064--2079, 2006.
\newblock \href {https://doi.org/10.1016/j.dam.2005.04.011}
  {\path{doi:10.1016/j.dam.2005.04.011}}.

\bibitem{Butkovic2010Maxlinear}
P.~Butkovi\v{c}.
\newblock {\em Max-linear Systems}.
\newblock Springer Monographs in Mathematics. Springer, London, 2010.
\newblock \href {https://doi.org/10.1007/978-1-84996-299-5}
  {\path{doi:10.1007/978-1-84996-299-5}}.

\bibitem{Cuninghamegreen1962Describing}
R.~A. Cuninghame-Green.
\newblock Describing industrial processes with interference and approximating
  their steady-state behaviour.
\newblock {\em Oper. Res. Quart.}, 13(1):95--100, 1962.
\newblock \href {https://doi.org/10.2307/3007584} {\path{doi:10.2307/3007584}}.

\bibitem{Demeulemeester2002Project}
E.~L. Demeulemeester and W.~S. Herroelen.
\newblock {\em Project Scheduling}, volume~49 of {\em International Series in
  Operations Research and Management Science}.
\newblock Springer, New York, 2002.
\newblock \href {https://doi.org/10.1007/b101924} {\path{doi:10.1007/b101924}}.

\bibitem{Giffler1963Scheduling}
B.~Giffler.
\newblock Scheduling general production systems using schedule algebra.
\newblock {\em Naval Res. Logist. Quart.}, 10(1):237--255, September 1963.
\newblock \href {https://doi.org/10.1002/nav.3800100119}
  {\path{doi:10.1002/nav.3800100119}}.

\bibitem{Golan2003Semirings}
J.~S. Golan.
\newblock {\em Semirings and Affine Equations Over Them}, volume 556 of {\em
  Mathematics and Its Applications}.
\newblock Kluwer Acad. Publ., Dordrecht, 2003.
\newblock \href {https://doi.org/10.1007/978-94-017-0383-3}
  {\path{doi:10.1007/978-94-017-0383-3}}.

\bibitem{Gondran2008Graphs}
M.~Gondran and M.~Minoux.
\newblock {\em Graphs, Dioids and Semirings}, volume~41 of {\em Operations
  Research/ Computer Science Interfaces}.
\newblock Springer, New York, 2008.
\newblock \href {https://doi.org/10.1007/978-0-387-75450-5}
  {\path{doi:10.1007/978-0-387-75450-5}}.

\bibitem{Goto2009Robust}
H.~Goto.
\newblock Robust {MPL} scheduling considering the number of in-process jobs.
\newblock {\em Eng. Appl. Artif. Intell.}, 22(4):603--607, 2009.
\newblock \href {https://doi.org/10.1016/j.engappai.2008.11.007}
  {\path{doi:10.1016/j.engappai.2008.11.007}}.

\bibitem{Goto2017Forward}
H.~Goto.
\newblock Forward-compatible framework with critical-chain project management
  using a max-plus linear representation.
\newblock {\em Opsearch}, 54(1):201--216, 2017.
\newblock \href {https://doi.org/10.1007/s12597-016-0276-3}
  {\path{doi:10.1007/s12597-016-0276-3}}.

\bibitem{Heidergott2006Maxplus}
B.~Heidergott, G.~J. Olsder, and J.~{van der Woude}.
\newblock {\em Max {P}lus at Work}.
\newblock Princeton Series in Applied Mathematics. Princeton Univ. Press,
  Princeton, NJ, 2006.

\bibitem{Karmarkar1984New}
N.~Karmarkar.
\newblock A new polynomial-time algorithm for linear programming.
\newblock {\em Combinatorica}, 4(4):373--395, 1984.
\newblock \href {https://doi.org/10.1007/BF02579150}
  {\path{doi:10.1007/BF02579150}}.

\bibitem{Krivulin2015Extremal}
N.~Krivulin.
\newblock Extremal properties of tropical eigenvalues and solutions to tropical
  optimization problems.
\newblock {\em Linear Algebra Appl.}, 468:211--232, 2015.
\newblock \href {https://arxiv.org/abs/1311.0442} {\path{arXiv:1311.0442}},
  \href {https://doi.org/10.1016/j.laa.2014.06.044}
  {\path{doi:10.1016/j.laa.2014.06.044}}.

\bibitem{Krivulin2017Direct}
N.~Krivulin.
\newblock Direct solution to constrained tropical optimization problems with
  application to project scheduling.
\newblock {\em Comput. Manag. Sci.}, 14(1):91--113, 2017.
\newblock \href {https://arxiv.org/abs/1501.07591} {\path{arXiv:1501.07591}},
  \href {https://doi.org/10.1007/s10287-016-0259-0}
  {\path{doi:10.1007/s10287-016-0259-0}}.

\bibitem{Krivulin2017Tropicaloptimization}
N.~Krivulin.
\newblock Tropical optimization problems in time-constrained project
  scheduling.
\newblock {\em Optimization}, 66(2):205--224, 2017.
\newblock \href {https://arxiv.org/abs/1502.06222} {\path{arXiv:1502.06222}},
  \href {https://doi.org/10.1080/02331934.2016.1264946}
  {\path{doi:10.1080/02331934.2016.1264946}}.

\bibitem{Krivulin2017Tropical}
N.~Krivulin.
\newblock Tropical optimization problems with application to project scheduling
  with minimum makespan.
\newblock {\em Ann. Oper. Res.}, 256(1):75--92, 2017.
\newblock \href {https://arxiv.org/abs/1403.0268} {\path{arXiv:1403.0268}},
  \href {https://doi.org/10.1007/s10479-015-1939-9}
  {\path{doi:10.1007/s10479-015-1939-9}}.

\bibitem{Krivulin2020Tropical}
N.~Krivulin.
\newblock Tropical optimization technique in bi-objective project scheduling
  under temporal constraints.
\newblock {\em Comput. Manag. Sci.}, 17(3):437--464, 2020.
\newblock \href {https://arxiv.org/abs/1907.09028} {\path{arXiv:1907.09028}},
  \href {https://doi.org/10.1007/s10287-020-00374-5}
  {\path{doi:10.1007/s10287-020-00374-5}}.

\bibitem{Krivulin2022Minimizing}
N.~Krivulin and S.~Sergeev.
\newblock Minimizing maximum lateness in two-stage projects by tropical
  optimization.
\newblock {\em Kybernetika}, 58(5):816--841, 2022.
\newblock \href {https://arxiv.org/abs/2108.06425} {\path{arXiv:2108.06425}},
  \href {https://doi.org/10.14736/kyb-2022-5-0816}
  {\path{doi:10.14736/kyb-2022-5-0816}}.

\bibitem{Krivulin2021Algebraicsolution}
N.~K. Krivulin and S.~A. Gubanov.
\newblock Algebraic solution of a problem of optimal project scheduling in
  project management.
\newblock {\em Vestnik St. Petersburg Univ. Math.}, 54(1):58--68, 2021.
\newblock \href {https://doi.org/10.1134/S1063454121010088}
  {\path{doi:10.1134/S1063454121010088}}.

\bibitem{Kroger2022General}
A.~Kroger, L.~Bahta, and P.~Hunter.
\newblock {\em General Best Practice Guidelines for Immunization}.
\newblock Centers for Disease Control and Prevention, Atlanta, GA, 2022.
\newblock Accessed 27 January 2023.

\bibitem{Maclagan2015Introduction}
D.~Maclagan and B.~Sturmfels.
\newblock {\em Introduction to Tropical Geometry}, volume 161 of {\em Graduate
  Studies in Mathematics}.
\newblock AMS, Providence, RI, 2015.
\newblock \href {https://doi.org/10.1090/gsm/161} {\path{doi:10.1090/gsm/161}}.

\bibitem{Neumann2003Project}
K.~Neumann, C.~Schwindt, and J.~Zimmermann.
\newblock {\em Project Scheduling with Time Windows and Scarce Resources}.
\newblock Springer, Berlin, 2 edition, 2003.
\newblock \href {https://doi.org/10.1007/978-3-540-24800-2}
  {\path{doi:10.1007/978-3-540-24800-2}}.

\bibitem{Pandit1961New}
S.~N.~N. Pandit.
\newblock A new matrix calculus.
\newblock {\em J. SIAM}, 9(4):632--639, 1961.
\newblock \href {https://doi.org/10.1137/0109052} {\path{doi:10.1137/0109052}}.

\bibitem{Department2006Immunisation}
D.~Salisbury, M.~Ramsay, and K.~Noakes, editors.
\newblock {\em Immunisation Against Infectious Diseases}.
\newblock The Stationery Office, London, 3 edition, 2006.

\bibitem{Singh2014Efficient}
M.~Singh and R.~P. Judd.
\newblock Efficient calculation of the makespan for job-shop systems without
  recirculation using max-plus algebra.
\newblock {\em Internat. J. Prod. Res.}, 52(19):5880--5894, 2014.
\newblock \href {https://doi.org/10.1080/00207543.2014.925600}
  {\path{doi:10.1080/00207543.2014.925600}}.

\bibitem{Tkindt2006Multicriteria}
V.~T{\textquoteright}kindt and J.-C. Billaut.
\newblock {\em Multicriteria Scheduling}.
\newblock Springer, Berlin, 2 edition, 2006.
\newblock \href {https://doi.org/10.1007/b106275} {\path{doi:10.1007/b106275}}.

\bibitem{Vanhoucke2012Project}
M.~Vanhoucke.
\newblock {\em Project Management with Dynamic Scheduling}.
\newblock Springer, Berlin, 2012.
\newblock \href {https://doi.org/10.1007/978-3-642-40438-2}
  {\path{doi:10.1007/978-3-642-40438-2}}.

\bibitem{Zimmermann2003Disjunctive}
K.~Zimmermann.
\newblock Disjunctive optimization, max-separable problems and extremal
  algebras.
\newblock {\em Theoret. Comput. Sci.}, 293(1):45--54, 2003.
\newblock \href {https://doi.org/10.1016/S0304-3975(02)00231-1}
  {\path{doi:10.1016/S0304-3975(02)00231-1}}.

\end{thebibliography}

\end{document}